\documentclass[10pt, reqno]{amsart}
\numberwithin{equation}{section}
\usepackage{amssymb}
\usepackage{hyperref}

\let\Re=\undefined\DeclareMathOperator*{\Re}{Re}

\DeclareMathOperator*{\wlim}{w-lim}

\newcommand{\R}{\mathbb{R}}
\newcommand{\C}{\mathbb{C}}
\newcommand{\E}{\mathcal{E}}
\newcommand{\wh}[1]{\widehat{#1}}

\newcommand{\norm}[1]{\| #1\|}
\newcommand{\bnorm}[1]{\bigg\|#1\bigg\|}
\newcommand{\eps}{\varepsilon}

\newcommand{\norml}[2]{\|#1\|_{L_x^{#2}}}

\newcommand{\N}{\mathcal{N}}

\newcommand{\M}{\mathcal{M}}
\renewcommand{\S}{\mathcal{S}}

\newcommand{\F}{\mathcal{F}}

\newtheorem{theorem}{Theorem}[section]

\newtheorem{lemma}[theorem]{Lemma}

\newtheorem{corollary}[theorem]{Corollary}

\newtheorem{proposition}[theorem]{Proposition}

\theoremstyle{definition}
\newtheorem{definition}[theorem]{Definition}
\newtheorem{remark}[theorem]{Remark}

\theoremstyle{remark}

\newcommand{\qtq}[1]{\quad\text{#1}\quad}

\def\({\left(}
\def\){\right)}
\def\fhsc{\F \dot{H}^{|s_c|}}
\newcommand{\cn}[1]{\| #1\|_{\fhsc}}

\begin{document}
\title[Scattering for mass-subcritical NLS]{Large data mass-subcritical NLS: \\ critical  weighted bounds imply scattering}
\author[R. Killip]{Rowan Killip}\address{Department of Mathematics, UCLA, Los Angeles, USA}\email{killip@math.ucla.edu}	
\author[S. Masaki]{Satoshi Masaki}\address{Department of Systems Innovation Graduate School of Engineering Science Osaka University, Toyonaka, Osaka 560-8531, Japan}\email{masaki@sigmath.es.osaka-u.ac.jp}
\author[J. Murphy]{Jason Murphy}\address{Department of Mathematics, University of California, Berkeley, USA}\email{murphy@math.berkeley.edu}
\author[M. Visan]{Monica Visan}\address{Department of Mathematics, UCLA, Los Angeles, USA}\email{visan@math.ucla.edu}

\maketitle

\begin{abstract}  We consider the mass-subcritical nonlinear Schr\"odinger equation in all space dimensions with focusing or defocusing nonlinearity.  For such equations with critical regularity $s_c\in(\max\{-1,-\frac{d}{2}\},0)$, we prove that any solution satisfying
\[
\|\, |x|^{|s_c|}e^{-it\Delta} u\|_{L_t^\infty L_x^2} <\infty
\]
on its maximal interval of existence must be global and scatter.
\end{abstract}

\section{Introduction}
We consider the initial-value problem for the following nonlinear Schr\"odinger equation (NLS) on $\R\times\R^d$:
\begin{equation}\label{nls}
(i \partial_t  + \Delta) u = \mu |u|^{p} u.
\end{equation}
Here $d\geq 1$, $p>0$, and $u:\R\times\R^d\to\C$ is a complex-valued function.  We consider $\mu=\pm 1$, corresponding to the defocusing and focusing equations, respectively.

The class of solutions to \eqref{nls} is invariant under the rescaling
\begin{equation}\label{eq:scale1}
u(t,x)\mapsto u_{[\lambda]}(t,x) := \lambda^{\frac2p} u(\lambda^2t,\lambda x)\qtq{for}\lambda>0,
\end{equation}
which on initial data takes the form
\begin{equation}\label{eq:scale1.5}
\phi(x)\mapsto \phi_{\{\lambda\}}(x):=\lambda^{\frac2p}\phi(\lambda x)\qtq{for}\lambda>0.
\end{equation}
The scaling symmetry defines a notion of \emph{criticality} for \eqref{nls}.  In particular, the only homogeneous $L_x^2$-based Sobolev space that is left invariant under \eqref{eq:scale1.5} is $\dot H_x^{s_c}$, where the \emph{critical regularity} $s_c$ is given by $s_c=\frac{d}{2}-\frac{2}{p}$.  

Solutions to \eqref{nls} also enjoy the conservation of \emph{mass} and \emph{energy}, defined by
\begin{align*}
&M(u) = \int_{\R^d} |u(t,x)|^2 \,dx, \\
&E(u) = \int_{\R^d} \tfrac12 |\nabla u(t,x)|^2 + \tfrac{\mu}{p+2}|u(t,x)|^{p+2}\,dx,
\end{align*}
respectively.  Equation \eqref{nls} with $p=\frac{4}{d}$ is known as the \emph{mass-critical} NLS, since in this case $M(u_{[\lambda]})\equiv M(u)$ (i.e., $s_c=0$).  Similarly, \eqref{nls} with $p=\frac{4}{d-2}$ (and $d\geq 3$) is known as the \emph{energy-critical} NLS, since in this case $E(u_{[\lambda]})\equiv E(u)$ (i.e., $s_c=1$).  The presence of a scale-invariant conserved quantity has contributed greatly to the popularity of these two particular models.  We review some of the main results concerning these cases below. 

In this paper, we consider the \emph{mass-subcritical} NLS, in which case the critical regularity $s_c$ is negative (or equivalently, $p<\frac4d$).  In this setting, we prescribe initial data in a weighted space, rather than a Sobolev space of negative regularity.  For data in negative regularity Sobolev spaces, there are well-posedness results available in the radial setting \cite{Hid, ChoHwaOza}; however, some forms of \emph{ill-posedness} hold for non-radial data in $H_x^{s}(\R^d)$ whenever $s<0$ \cite{ChrColTao}. The scattering theory for mass-subcritical NLS with data in weighted spaces is a rich subject with an extensive literature, which we briefly review below.  This problem has also been studied in other scale-invariant spaces, such as Fourier--Lebesgue and hat-Morrey spaces \cite{Mas4}. 

In this paper, we study \eqref{nls} in weighted spaces. For $s\in[0,\frac{d}{2})$, we define $\F\dot H^s(\R^d)$ to be the completion of test functions under the norm
\[
\|f\|_{\F\dot H^s(\R^d)} :=  \| \F f \|_{\dot H^s(\R^d)}= \|\, |x|^s f\|_{L_x^2(\R^d)}.
\]
Here $\F$ denotes the Fourier transform.  For $t_0\in\R$, we consider equation \eqref{nls} with the initial condition
\begin{equation}\label{eq:IC}
e^{-i t_0 \Delta} u(t_0) = u_0\in\fhsc,
\end{equation}
where $e^{it\Delta}:=\F^{-1}e^{-it|\xi|^2}\F$ is the Schr\"odinger group.  We also allow $t_0=-\infty$, in which case we interpret \eqref{eq:IC} as
\begin{equation}\nonumber
	\cn{e^{-it\Delta}u(t)-u_0} \to 0\qtq{as} t\to-\infty.
\end{equation}

Equation \eqref{nls} with initial condition \eqref{eq:IC} is a \emph{critical} problem, in the sense that the $\fhsc$-norm is invariant under the rescaling \eqref{eq:scale1.5}.  

As in \cite{Mas2}, we work with the following notion of solution.

\begin{definition}[Solution]\label{D:solution} Let $I\subset\R$ be an interval. A function $u:I\times\R^d\to\C$ is a \emph{solution} to \eqref{nls} on $I$ if $e^{-it\Delta}u(t)\in C\bigl(I;\fhsc(\R^d)\bigr)$ and the following Duhamel formula holds for any $t,\tau\in I$:
\[
e^{-it\Delta}u(t) = e^{-i\tau\Delta}u(\tau) - i \mu\int_{\tau}^t e^{-is\Delta}\bigl(|u|^pu\bigr)(s)\,ds\quad\text{in }\fhsc.
\] 
We call $I$ the \emph{lifespan} of $u$. We call $u$ a \emph{maximal-lifespan} solution if it cannot be extended to any strictly larger interval. We denote the maximal interval of existence of $u$ by $I_{\max}=(T_{\min},T_{\max})$.  We call $u$ \emph{forward-global} if $T_{\max}=\infty$.
\end{definition}

\begin{remark}\label{R:tt} With this notion of solution, equation \eqref{nls} is no longer time-translation invariant.  More precisely, if $u$ is a solution to \eqref{nls} and $\tau\in\R$, then $u(\cdot+\tau)$ is not necessarily a solution.  
\end{remark}

In this paper, we consider all dimensions $d\geq 1$ and restrict the power of the nonlinearity as follows: 
\begin{equation}\label{asmp:p}
\max\bigl(\tfrac2d,\tfrac4{d+2} \bigr)<p< \tfrac4d.
\end{equation}
For this range, local well-posedness in $\fhsc$ was established in \cite{Mas2}; specifically, for any $t_0\in [-\infty,\infty)$ and $u_0\in\fhsc$, there exists an interval $I\subset\R$ with $t_0\in \bar{I}$ and a solution $u$ on $I$ satisfying \eqref{eq:IC}.  The restriction $p>\frac{4}{d+2}$ is technical---it guarantees $|s_c|<1$.  The restriction $p>\frac{2}{d}$, which is known as the \emph{short-range} case, is natural and represents the threshold for (unmodified) scattering, cf. \cite{Bar, Str, TsuYaj}. 

Our main result is the following.

\begin{theorem}\label{T:main} Assume \eqref{asmp:p}.  Let $t_0\in[-\infty,\infty)$ and $u_0\in\fhsc$.  Let $u:I_{\max}\times\R^d\to\C$ be the maximal-lifespan solution to \eqref{nls} with initial condition \eqref{eq:IC}.  Suppose
\begin{equation}\label{asmp:bdd}
\sup_{t\in I_{\max}} \norm{e^{-it\Delta} u(t)}_{\fhsc}  <\infty.
\end{equation}
Then $u$ is forward-global and scatters forward in time; that is, there exists $u_+\in\fhsc$ such that
\begin{equation}\label{eq:scatter}
\lim_{t\to\infty} \| e^{-it\Delta}u(t)-u_+\|_{\fhsc} = 0.
\end{equation}
\end{theorem}

\begin{remark}\label{R:t0} Because the time-translation symmetry is broken (cf. Remark~\ref{R:tt}), one cannot use it to renormalize $t_0=1$, say.  However, one can use the scaling symmetry of the problem to reduce Theorem~\ref{T:main} to a well-posedness statement about $t_0=1$.  We discuss this issue further in Section~\ref{S:existence}. 
\end{remark}

Contraction mapping arguments show that both \eqref{asmp:bdd} and \eqref{eq:scatter} hold for small data $u_0\in\fhsc$, as we discuss in Section~\ref{S:notation}.  The novelty of Theorem~\ref{T:main} appears in the case of large data.

To prove Theorem~\ref{T:main}, we will actually prove a stronger result, namely, that one has uniform space-time bounds depending only on LHS\eqref{asmp:bdd}.   It is not difficult to see that for individual solutions, space-time bounds imply scattering and conversely that scattering implies space-time bounds.  The key advantage of working with space-time bounds, versus scattering, is the uniformity of the space-time bounds with respect to the \emph{a priori} bound \eqref{asmp:bdd}.  

In the case of non-negative critical regularity, scattering is usually written in the equivalent form
\[
\lim_{t\to\infty} \|u(t)-e^{it\Delta}u_+\|_{\dot H_x^s}=0,
\]
with the clear interpretation that $u$ behaves asymptotically like a solution to the linear Schr\"odinger equation.  It is not obvious whether or not \eqref{eq:scatter} is equivalent to
\[
\lim_{t\to\infty} \| u(t)-e^{it\Delta}u_+\|_{\fhsc} = 0;
\]
we refer the interested reader to \cite{Beg}, where some positive results in this direction are proved under additional assumptions on the parameters $d$ and $p$. 

Theorem~\ref{T:main} fits in the context of recent work on NLS at critical regularity.  Beginning with the pioneering work of Bourgain \cite{Bou1} on the energy-critical NLS, the last 10--20 years have seen an explosion of progress in this area.  For the defocusing mass- and energy-critical problems, it is now known that arbitrary initial data in the critical Sobolev space lead to global solutions that scatter; for the focusing problems, scattering below the ground state is known in all cases except for the energy-critical problem in three dimensions \cite{Bou1, CKSTT, Dod1, Dod2, Dod3, Dod4, Dod5, Gri, KenMer, KTV, KV2, KV3, KVZ, RV, Tao1, TVZ, Vis0, Vis1, Vis2}. Concomitant with these results was the development of the concentration compactness approach to induction on energy \cite{BahGer, BegVar, Bou2, CarKer, KenMer, Ker1, Ker2, MerVeg}, which reduces the question of global well-posedness and scattering to the exclusion of `minimal counterexamples' satisfying good compactness properties.  See especially \cite{KenMer} for the first application of this approach to establish a global well-posedness result. 

Apart from the mass- and energy-critical cases, one does not have a (coercive) conserved quantity at the level of the critical regularity.  The limits of current technology yield a class of conditional results, in which an \emph{a priori} boundedness assumption such as \eqref{asmp:bdd} fulfills the role of the missing conservation law.  Scattering results under the assumption of $\dot H^{s_c}$-bounds have been established in both the \emph{energy-supercritical} ($s_c>1$) and \emph{intercritical} ($s_c\in(0,1)$) regimes \cite{DMMZ, KenMer2, KV4, MMZ, Mur1, Mur2, Mur3, XieFan}.  Theorem~\ref{T:main} is the first result of this type in the setting of mass-subcritical NLS; it is also the first such result for a focusing NLS. 

Theorem~\ref{T:main} should also be considered in the context of the scattering theory for mass-subcritical NLS in weighted spaces.  The literature here is quite extensive, and we only mention a few relevant results.  Previous works have considered mass-subcritical NLS with data $u_0\in\Sigma:=\dot H^1\cap \F\dot H^{1}$.  In particular, $u_0\in L_x^2$ and hence (by mass-subcriticality) solutions are global, even in the focusing case.  In the defocusing case, solutions scatter in $L_x^2$ if and only if $p>\frac{2}{d}$ \cite{Bar, Str, TsuYaj}.  For $p\geq \frac{4}{d+2}$, one expects scattering in $\Sigma$; however, for $p$ close to $\frac{4}{d+2}$, this is only known for small data \cite{CazWei}.  In the focusing case, solitons exist and provide examples of non-scattering solutions.  As we discuss below, solitons do not obey \eqref{asmp:bdd}.  Theorem~\ref{T:main} provides a sharp condition for scattering in $\fhsc$ for the range \eqref{asmp:p} for both the focusing and defocusing problems.  For a more extensive review of the literature, we refer the reader to \cite{Caz, NakOza} and the references therein.

Theorem~\ref{T:main} is a conditional result.  It relies on the \emph{a priori} assumption \eqref{asmp:bdd}.  This assumption is rather strong, as can be seen from the following estimate: by Sobolev embedding and \eqref{def:Xn} below,
\begin{equation}\label{strong-hyp}
\| u(t)\|_{L_x^{\frac{2d}{d-2|s_c|}}} \lesssim |t|^{-|s_c|}\|e^{-it\Delta}u(t)\|_{\fhsc}.
\end{equation}
Thus, \eqref{asmp:bdd} implies decay in time; in particular, \eqref{asmp:bdd} is inconsistent with solitons, which is already a big step in the direction of scattering.  This is in contrast to the case $s_c\geq 0$, in which case \emph{a priori} bounds on the $\dot H_x^{s_c}$-norm do not immediately imply any decay in time.

The decay \eqref{strong-hyp} is inherently incompatible with soliton solutions, thus showing that \eqref{asmp:bdd} is also not obeyed by soliton solutions.  This does not undermine the significance of Theorem~\ref{T:main} in the focusing case, since as discussed in \cite{Mas1, Mas2}, solitons are \emph{not} the minimal non-scattering solutions (by any metric) in the mass-subcritical setting.

In light of \eqref{strong-hyp}, it is natural to ask  whether \eqref{asmp:bdd} implies scattering rather immediately. On the one hand, the decay rate in \eqref{strong-hyp} matches the dispersive estimate for linear solutions.  Furthermore, \eqref{strong-hyp} implies
\begin{equation}\label{strong-hyp2}
\| u\|_{L_t^{\frac{1}{|s_c|},\infty} L_x^{\frac{2d}{d-2|s_c|}}(I\times\R^d)} \lesssim \|e^{-it\Delta}u(t)\|_{L_t^\infty \fhsc(I\times\R^d)}
\end{equation}
on any interval of existence, providing uniform control over a space-time norm that is invariant under the rescaling \eqref{eq:scale1}.  However, \emph{a priori} control over the norm appearing on the left-hand side of \eqref{strong-hyp2} alone is \emph{not} sufficient to yield scattering. Because this norm is of weak-type in time, one cannot break a time interval into subintervals on which this norm is small; such `frangibility' of a scaling-critical norm is crucial for deducing scattering from space-time bounds.  

We will prove that \eqref{asmp:bdd} \emph{does} imply control over a critically-scaling Lorentz-type norm of the form $L_t^{q,2}L_x^r$, which is a more suitable `scattering norm' (see the definition of the $S$-norm in Section~\ref{S:FS}).  This norm is finite for linear solutions with data in $\fhsc$ (see Proposition~\ref{prop:Strichartz} below); however, it is far from obvious that this norm is bounded for solutions satisfying the assumption \eqref{asmp:bdd}.  The proof of this fact essentially comprises the bulk of the paper.

As in previous works on NLS at critical regularity, we take the concentration compactness approach to induction on energy.  We show that if Theorem~\ref{T:main} were to fail, then we could find a minimal non-scattering solution.  As a consequence of minimality, this solution satisfies a compactness property, namely, the interaction variable $e^{-it\Delta}u(t)$ is almost periodic modulo symmetries in $\fhsc$.  In particular, the behavior of the solution is governed by time-dependent `modulation parameters' representing the physical dimension and velocity of the wave.  In Section~\ref{S:reduction}, we are able to further reduce consideration to the case of a certain type of `self-similar' solution (see Definition~\ref{Def:SS} below).

In order to rule out the self-similar scenario, we employ arguments inspired by earlier work on the radial mass-critical NLS \cite{KTV, KVZ}.  We show that such a solution would belong to $L_x^2$, which, by mass-subcriticality, implies that the solution would be global.  On the other hand, self-similar solutions blow up at $t=0$.  Thus, we reach a contradiction and complete the proof of Theorem~\ref{T:main}. 

The rest of this paper is organized as follows:  In Section~\ref{S:notation} we set up notation and record some useful lemmas.  In Section~\ref{S:existence} we prove Theorem~\ref{thm:apsol}, which shows that if Theorem~\ref{T:main} fails, then there exist almost periodic minimal blowup solutions.  We also discuss the role of $t_0$ in Theorem~\ref{T:main}.  In Section~\ref{S:reduction}, we show that if Theorem~\ref{T:main} fails, then we can find a self-similar almost periodic solution.  Finally, in Section~\ref{S:preclusion}, we show that self-similar almost periodic solutions cannot exist.

\subsection*{Acknowledgements} R. K. was supported by NSF grants DMS-1265868 and DMS-1600942 and by the Simons Foundation grant 342360.  S. M. was supported by JSPS KAKENHI, Scientic Research (S) 23224003 and JSPS KAKENHI, Grant-in-Aid for Young Scientists (B) 2474010. J. M. was supported by the NSF Postdoctoral Fellowship DMS-1400706. M. V. was supported by NSF grant DMS-1500707.

\section{Notation and useful lemmas}\label{S:notation}

For non-negative $X$ and $Y$, we write $X \lesssim Y$ to denote $X \le CY$ for some $C > 0$.  If $X \lesssim Y \lesssim X$, we write $X \sim Y$.  The dependence of implicit constants on parameters will be indicated by subscripts, e.g.\ $X \lesssim_u Y$ denotes $X \le CY$ for some $C = C(u)$.  We write $a'\in[1,\infty]$ to denote the H\"older dual exponent to $a\in[1,\infty]$, that is, the solution to $\tfrac{1}{a}+\tfrac{1}{a'}=1$. 

The Fourier transform on $\R^d$ is defined by
\[
\F f(\xi):=\wh{f}(\xi) := (2\pi)^{-\frac{d}{2}} \int_{\R^d} e^{-ix\xi} f(x)\,dx.
\]

For $s\in\R$, the operator $|\nabla|^s$ is defined to be the Fourier multiplier operator with multiplier $|\xi|^s$. 

We recall the standard Littlewood--Paley projection operators.  Let $\varphi$ be a radial bump function on $\R^d$ supported on $\{ |\xi| \le \frac53 \}$ and equal to one on $\{ |\xi| \le \frac43 \}$.  For $N\in 2^{\mathbb{Z}}$, we define the operators $P_N$, $P_{\leq N}$, and $P_{>N}$ via
\begin{align*}
&\widehat{P_{\le N} f}(\xi):= \widehat{f_{\le N}}(\xi):= \varphi(\tfrac{\xi}{N}) \widehat{f}(\xi) ,
\quad \widehat{P_{> N} f}(\xi):= \widehat{f_{> N}}(\xi):= \bigl(1-\varphi(\tfrac{\xi}{N})\bigr) \widehat{f}(\xi) ,\\
&\widehat{P_{N} f}(\xi):= \widehat{f_{N}}(\xi):= \bigl(\varphi(\tfrac{\xi}{N})-\varphi(\tfrac{2\xi}{N})\bigr)\widehat{f}(\xi).
\end{align*}

We recall the following standard Bernstein estimates. 
\begin{lemma}[Bernstein] For $1 \le r \le q \le \infty$,
\begin{align*}
\norml{|\nabla|^s f_N}r &\sim N^s \norml{f_N}r \qtq{for}s\in\R, \\
\norml{|\nabla|^s f_{\le N}}r  &\lesssim N^s \norml{f_{\le N}}r \qtq{for}s\geq 0,\\
\norml{ f_{\le N}}q  &\lesssim N^{\frac{d}r-\frac{d}q} \norml{ f_{\le N}}r.
\end{align*}
\end{lemma}

We also need the following fractional calculus estimates.  The first is a fractional chain rule from \cite{ChrWei}.  The second is a fractional chain rule for H\"older continuous functions from \cite{Vis1}; see also \cite[Lemma 3.7]{MasSeg} for the case where $G$ is $C^{1+\alpha}$ and $0<s<1+\alpha$. 

\begin{lemma}[Fractional calculus \cite{ChrWei, Vis1}]\label{L:fractional} \text{ }
\begin{itemize}
\item[(i)] Suppose $G\in C^1(\C)$ and $s\in(0,1]$. Let $1<r,r_2<\infty$ and $1<r_1\leq\infty$ satisfy $\frac{1}{r}=\frac{1}{r_1}+\frac{1}{r_2}$.  Then
\[
\| |\nabla|^s G(u) \|_{L_x^r} \lesssim \|G'(u)\|_{L_x^{r_1}} \||\nabla|^s u\|_{L_x^{r_2}}. 
\]
\item[(ii)] Suppose $G$ is H\"older continuous of order $\alpha\in(0,1)$. Then, for every $0<s<\alpha$, $1<r<\infty$, and $\frac{s}{\alpha}<\sigma<1$, we have
\[
\| |\nabla|^s G(u) \|_{L_x^r} \lesssim \|\,|u|^{\alpha-\frac{s}{\sigma}}\|_{L_x^{r_1}} \| |\nabla|^\sigma u\|_{L_x^{\frac{s}{\sigma}r_2}}^{\frac{s}{\sigma}},
\]
provided $\frac{1}{r}=\frac{1}{r_1}+\frac{1}{r_2}$ and $(1-\frac{s}{\alpha\sigma})r_1>1$.
\end{itemize}
\end{lemma}	

We also make use of the following paraproduct estimate.  This type of estimate was proven in \cite{Vis2} in four dimensions; it is straightforward to generalize the proof to any dimension.

\begin{lemma}[Paraproduct estimate, \cite{Vis2}]\label{L:para}
Let $1<r<r_1<\infty$, $1<r_2<\infty$, and $0<s<d$ satisfy 
$\frac{d}{r_1}+\frac{d}{r_2} = \frac{d}r + s <d$.
Then,
\[
	\norml{|\nabla|^{-s}(fg)}r \lesssim \norml{|\nabla|^{-s}f}{r_1} \norml{|\nabla|^{s} g}{r_2}.
\]
\end{lemma}

\subsection{The linear Schr\"odinger equation} The Schr\"odinger group $e^{it\Delta}=\F^{-1}e^{-it|\xi|^2}\F$ is written in physical space as
\[
[e^{it\Delta} f](x) = (4\pi i t)^{-\frac{d}{2}}\int_{\R^d} e^{\frac{i|x-y|^2}{4t}} f(y)\,dy, \quad t\neq 0.
\]
From this, one can read off the $L_x^1\to L_x^\infty$ dispersive estimate;  as $e^{it\Delta}$ is unitary on $L_x^2$, interpolation yields the following general class of dispersive estimates:
\[
\| e^{it\Delta} f\|_{L_x^r(\R^d)} \lesssim |t|^{-(\frac{d}{2}-\frac{d}{r})} \|f\|_{L_x^{r'}(\R^d)}\qtq{for all}r\geq 2.
\]

Defining the multiplication operator $M(t):=e^{\frac{i|x|^2}{4t}}$ and the dilation operator
\[
[D(t)f](x) := (2i t)^{-\frac{d}{2}} f\bigl(\tfrac{x}{2t}\bigr),
\]
one can also read off the well-known factorization $e^{it\Delta} = M(t)D(t)\F M(t)$. From this, one can deduce the identity
\begin{equation}\label{eq:d_form}
e^{it\Delta} m(x) e^{-it\Delta}  = M(t) m(2it\nabla) M(-t)
\end{equation}
for reasonable multipliers $m$, where we write $m(i\nabla)$ for the Fourier multiplier operator with multiplier $m(\xi)$. 

The Galilean operator $J(t):=x+2it\nabla$ has played an important role in the scattering theory for mass-subcritical NLS.  By direct calculation and \eqref{eq:d_form}, we may also write
\[
J(t) = e^{it\Delta} x e^{-it\Delta} =  M(t) 2it\nabla M(-t).
\]

We define powers of $J$ as follows:
\[
J^s(t) = e^{it\Delta}|x|^s e^{-it\Delta} = M(t) (-4t^2\Delta)^{\frac{s}{2}} M(-t)\qtq{for}s\in\R.
\]
It will be useful to have localized versions of the operators $J^s(t)$, as well.  Using \eqref{eq:d_form}, we define
\begin{equation}\label{def:JN}
J_{\leq N}^s(t) := e^{it\Delta} \varphi(\tfrac{x}{N})|x|^s e^{-it\Delta} = M(t) P_{\leq \frac{N}{|t|}}(-4t^2\Delta)^{\frac{s}{2}}M(-t).
\end{equation}
We can similarly define $J_{N}^s$ and $J_{>N}^s$. 

The identity
\begin{equation}\label{gal1}
[e^{it\Delta}(e^{ix\xi_0}f)](x) = e^{-it|\xi_0|^2+ix\xi_0}(e^{it\Delta}f)(x-2t\xi_0)
\end{equation}
shows that the class of solutions to the linear Schr\"odinger equation is invariant under \emph{Galilei boosts}:
\begin{equation}\label{gsym}
u(t,x) \mapsto e^{-it|\xi_0|^2+ix\xi_0}u(t,x-2\xi_0 t),\quad \xi_0\in\R^d.
\end{equation}
Exploiting gauge invariance, one can check that \eqref{gsym} is a symmetry for the nonlinear equation \eqref{nls}, as well.

Incorporating scaling into \eqref{gal1} leads to the useful formula
\begin{equation}\label{gal}
\bigl[e^{it\Delta}\bigl(e^{ix\xi_0}f(\tfrac{\cdot}{\lambda})\bigr)\bigr](x) = e^{-it|\xi_0|^2+ix\xi_0}(e^{it\lambda^{-2}\Delta}f)(\tfrac{x-2t\xi_0}{\lambda}). 
\end{equation}

\subsection{Function spaces} The operators $J^s(t)$ essentially play the role of differentiation operators.  We define the time-dependent spaces $\dot X^{s,r}=\dot X^{s,r}(t)$ via the norm
\begin{equation}\label{def:Xn}
\norm{f}_{\dot{X}^{s,r}} := \norm{ J^s(t) f}_{L_x^r(\mathbb{R}^d)}\sim \norm{\,|t|^s |\nabla|^s M(-t) f}_{L_x^r(\R^d)}.
\end{equation}
When $r=2$, we abbreviate $\dot X^s = \dot X^{s,2}$.  The equivalence of norms in \eqref{def:Xn} is valid provided $t\neq 0$. 

At certain points in the paper, it will be natural to write
\begin{equation}\label{def:sp}
f\in e^{it\Delta}\F\dot H^s\iff e^{-it\Delta}f\in\F\dot H^s.
\end{equation}
This is simply a change of notation: $e^{it\Delta}\F\dot H^s=\dot X^{s}(t)$.

We use Lorentz-modified space-time norms.  For an interval $I$, $1\leq q<\infty$, and $1\leq\alpha\leq\infty$, the \emph{Lorentz space} $L_t^{q,\alpha}(I)$ is defined via the quasi-norm
\[
\| f\|_{L_t^{q,\alpha}(I)} := \bigl\| \lambda\,\bigl|\{t\in I: |f(t)|>\lambda\}\bigr|^{\frac{1}{q}}\bigr\|_{L^\alpha((0,\infty),\frac{d\lambda}{\lambda})}.
\]
For a Banach space $X$, we write $L_t^{q,\alpha}X(I\times\R^d)$ for the space of functions $u:I\times\R^d\to\C$ such that
\[
\| u\|_{L_t^{q,\alpha}X(I\times\R^d)} := \bigl\|\,\|u(t)\|_{X} \bigr\|_{L_t^{q,\alpha}(I)}<\infty.
\]

The following H\"older's inequality for Lorentz spaces will be of frequent use.

\begin{lemma}[H\"older in Lorentz spaces, \cite{One,Hunt}]\label{lem:gH} The following estimate holds:
\[
\norm{fg}_{L_t^{q,\alpha}}\lesssim \norm{f}_{L_t^{q_1,\alpha_1}} \norm{g}_{L_t^{q_2,\alpha_2}}
\]
for any $1\leq q,q_1,q_2<\infty$ and $1\leq \alpha,\alpha_1,\alpha_2\leq\infty$ satisfying
\[
\tfrac{1}{q}=\tfrac{1}{q_1}+\tfrac{1}{q_2}\qtq{and} \tfrac{1}{\alpha}=\tfrac{1}{\alpha_1}+\tfrac{1}{\alpha_2}.
\]
\end{lemma}

We also have the following equivalence:
\[
\| f\|_{L_t^{q,\alpha}(I)} \sim \bigl\| \, \|f \cdot\textbf{1}_{\{2^{k-1}\leq |f|\leq2^{k}\}}\|_{L_t^q(I)}\bigr\|_{\ell^\alpha(k\in\mathbb{Z})}.
\]
In particular, if $f$ is roughly constant on $I$, then $\|f\|_{L_t^{q,\alpha}(I)}\sim\|f\|_{L_t^q(I)}$.

\subsection{Strichartz estimates} We utilize Strichartz estimates adapted to the spaces $L_t^{q,\alpha}\dot X^{s,r}$.  The standard Strichartz estimates for $e^{it\Delta}$ were proved in \cite{Strichartz, KeeTao, GinVel}.  The specific variants we need were proved in \cite{NakOza}.

\begin{definition} A pair $(q,r)$ is \emph{admissible} if 
\[
2<q<\infty,\quad 2<r<\begin{cases} \frac{2d}{d-2} & d\geq 3 \\ \infty & d\in\{1,2\}\end{cases},\qtq{and} \tfrac{2}{q}+\tfrac{d}{r}=\tfrac{d}{2}.
\]
Note that according to this definition, endpoints are \emph{not} admissible.
\end{definition}

\begin{proposition}[Strichartz estimates, \cite{NakOza}]\label{prop:Strichartz}
Let $s\geq 0$ and $t_0 \in I \subset \mathbb{R}$.
\begin{enumerate}
\item For any admissible pair $(q,r)$, 
\[
\norm{e^{it\Delta} f}_{L_t^\infty \dot{X}^s \cap L_t^{q,2}\dot X^{s,r}(\mathbb{R}\times\mathbb{\R}^d)} \lesssim \norm{f}_{\mathcal{F} \dot{H}^s}.
\]
\item For any admissible pairs $(q,r)$ and $(\alpha,\beta)$,
\[
\bnorm{ \int_{t_0}^t e^{i(t-s)\Delta} F(s) \,ds}_{L_t^\infty\dot{X}^s\cap L_t^{q,2}\dot X^{s,r}(I\times \R^d)}\lesssim \norm{F}_{L_t^{\alpha',2}\dot{X}^{s,\beta'}(I \times \R^d)}.
\]
\end{enumerate}
\end{proposition}

\subsection{Specific function spaces}\label{S:FS} Throughout this paper, we use some specific choices of exponents repeatedly.  These are the same exponents that were used in \cite{Mas2}.  Henceforth, we assume \eqref{asmp:p}.

If $p\geq 1$, we define
\[
\bigl(\tfrac1{q_1},\tfrac1{r_1} \bigr) = \bigl(\tfrac1{q_2},\tfrac1{r_2} \bigr) := \bigl(\tfrac{dp-2}{2(p+2)}, \tfrac{2d+4-dp}{2d(p+2)} \bigr),
\quad\bigl( \tfrac1{\tilde{q}}, \tfrac1{\tilde{r}} \bigr):=\bigl(\tfrac{4-(d-1)p}{p(p+2)}, \tfrac{2dp-4}{dp(p+2)}\bigr).
\]

For $p<1$, we define
\begin{align*}
s_0&:=\tfrac2{p}+\tfrac{d(d-2)}8 (p+1) -\tfrac{d^2+6d-8}{8}, \\
\bigl(\tfrac1{q_1},\tfrac1{r_1} \bigr) &:= \bigl( \tfrac{dp}8-\tfrac{p(s_0+1)}{4(p+1)} ,\tfrac12- \tfrac{p}{4} + \tfrac{p(s_0+1)}{2d(p+1)} \bigr),\\
\bigl( \tfrac1{\tilde{q}}, \tfrac1{\tilde{r}} \bigr) &:= \bigl( \tfrac1p -\tfrac{d}4+\tfrac{s_0+1}{2(p+1)} ,\tfrac12- \tfrac{s_0+ 1}{d(p+1)}\bigr), \\
\bigl(\tfrac1{q_2},\tfrac1{r_2} \bigr) &:= \bigl( \tfrac1{\tilde{q}} -|s_c|,\tfrac1{\tilde{r}} +\tfrac{|s_c|}d\bigr).
\end{align*}
Notice that \eqref{asmp:p} guarantees $0< s_0 < p |s_c|$.

The pairs $(q_j,r_j)$ are admissible; the pair $(\tilde q, \tilde r)$ satisfies the critical scaling relation
\[
\tfrac{2}{\tilde q}+\tfrac{d}{\tilde r} = \tfrac{d}{2}-s_c=\tfrac{2}{p}.
\]
The scaling relations between these exponents are as follows: 
\begin{equation}\label{scaling}
 \tfrac{1}{q_1'}=\tfrac{p}{\tilde q} + \tfrac{1}{q_1}, \quad \tfrac{1}{r_1'}=\tfrac{p}{\tilde r}+ \tfrac{1}{r_1}, \quad \tfrac{1}{\tilde q} - \tfrac{1}{q_2} = \tfrac{d}{r_2} - \tfrac{d}{\tilde r} = |s_c|.
\end{equation}

We define the spaces
\[
 S^{\text{weak}} := L_t^{\tilde q,\infty} L_x^{\tilde r},\quad S:= L_t^{\tilde q,2} L_x^{\tilde r}, \quad W := \cap_{j=1}^2 L_t^{q_j,2}\dot X^{|s_c|,r_j}
\]
for the solution, as well as the space
\[
N := L_t^{q_1',2}\dot X^{|s_c|,r_1'}
\]
for the nonlinearity.  We write $S^{\text{weak}}(I)$ to indicate that the norm is taken over the space-time slab $I\times\R^d$, and similarly for the other spaces.  We also use the notation
\[
S_I(u) := \|u\|_{S(I)}. 
\]

\subsection{Nonlinear estimates} In this section, we collect estimates that will be crucial for controlling the nonlinearity throughout the rest of the paper.  Throughout the section, we denote
\[
\tilde u(t):=M(-t)u(t).
\]

\begin{lemma}[Embeddings]\label{L:embeddings} The following holds:
\[
\|u\|_{S^{\text{weak}}} \lesssim \|u\|_{S} \lesssim \|u\|_{W}. 
\]
\end{lemma}

\begin{proof} The first estimate is simply the embedding $L_t^{\tilde q,2}\hookrightarrow L_t^{\tilde q,\infty}$.  For the second estimate, we use Sobolev embedding, \eqref{scaling}, and H\"older  to estimate
\begin{align*}
\|u\|_{L_t^{\tilde q,2} L_x^{\tilde r}} & \lesssim \| |\nabla|^{|s_c|}\tilde u\|_{L_t^{\tilde q,2} L_x^{r_2}} \lesssim \| |t|^{-|s_c|} \|_{L_t^{\frac{1}{|s_c|},\infty}} \| |t|^{|s_c|}|\nabla|^{|s_c|}\tilde u\|_{L_t^{q_2,2} L_x^{r_2}}.
\end{align*}
The result now follows from \eqref{def:Xn}.\end{proof}

\begin{lemma}[Nonlinear estimate]\label{L:nonlinear} The following holds:
\[
\|\,|u|^p u\|_{N} \lesssim \|u\|_{S^{\text{weak}}}^p \|u\|_W. 
\]
\end{lemma}

\begin{proof}  By gauge invariance and \eqref{def:Xn}, we can write
\[
\| J^{|s_c|}\bigl(|u|^p u\bigr)\|_{L_t^{q_1',2}L_x^{r_1'}} \sim \|\, |t|^{|s_c|}|\nabla|^{|s_c|}\bigl(|\tilde u|^p \tilde u\bigr)\|_{L_t^{q_1',2} L_x^{r_1'}}.
\]
The result now follows from the fractional chain rule, H\"older, \eqref{scaling}, and \eqref{def:Xn}. \end{proof} 

The following technical estimate is a consequence of the paraproduct estimate (Lemma~\ref{L:para}) and the fractional chain rule.  It will be used in Section~\ref{S:preclusion}. 

\begin{lemma}\label{L:parap} Let $F(z)=|z|^p z$. For $\eps>0$ sufficiently small,
\begin{equation}\label{E:parap}
\bigl\| |t\nabla|^{-\eps} [F(\tilde u+\tilde v) - F(\tilde u)] \bigr\|_{L_t^{q_1',2}L_x^{r_1'}}
	\lesssim \Bigl\{ \| u\|_{W}^p + \| v\|_{W}^p\Bigr\}  \| |t\nabla|^{-\eps}\tilde v\|_{L_t^{q_1,2}L_x^{r_1}}.
\end{equation}
\end{lemma}

\begin{proof} Let $\tilde u_\theta := \tilde u+\theta\tilde v$.  By the Fundamental Theorem of Calculus, we have
$$
F(\tilde u+\tilde v) - F(\tilde u) = \int_0^1 G_1(\tilde u_\theta) \tilde v + G_2(\tilde u_\theta) \overline{\tilde v} \,d\theta,
$$
where
$$
G_1(z)= \tfrac{p+2}{2} |z|^p \qtq{and} G_2(z)= \tfrac{p z^2}{2 |z|^2} |z|^p .
$$
Notice that both are H\"older continuous of order $p$ when $p<1$.

By Lemma~\ref{L:para} and \eqref{scaling}, we can estimate
\begin{align*}
\text{LHS\eqref{E:parap}} \lesssim \| |t \nabla|^{-\eps} \tilde v\|_{L_t^{q_1,2}L_x^{r_1}}
	\int_0^1 \sum_{k=1}^2 \ \bigl\| |\nabla|^{\eps} G_k(\tilde u_\theta) \bigr\|_{L_t^{\frac{\tilde q}{p},\infty} L_x^{\gamma}} \ d\theta
\end{align*}
for $\eps>0$ small enough, where we take $\tfrac{d}{\gamma}=\tfrac{d}{r_1'} - \tfrac{d}{r_1} + \eps.$

If $p\geq 1$, we use the fractional chain rule, Lemma~\ref{L:embeddings} (and its proof), \eqref{scaling}, and Sobolev embedding to estimate
\begin{align*}
\| |\nabla|^\eps G_k(\tilde u_\theta)\|_{L_t^{\frac{\tilde q}{p},\infty} L_x^{\gamma}}
&\lesssim  \|\tilde u_\theta\|_{L_t^{\tilde q,\infty}L_x^{\tilde r}}^{p-1} \||\nabla|^{\eps}\tilde u_\theta\|_{L_t^{\tilde q,\infty} L_x^{\tilde\gamma}}  \\
&\lesssim \|u_\theta \|_{W}^{p-1} \| |\nabla|^{|s_c|}\tilde u_\theta\|_{L_t^{\tilde q,\infty} L_x^{r_2}}  \lesssim \|u_\theta \|_{W}^p,
\end{align*}
where we take $\tfrac{1}{\tilde\gamma} = \tfrac{1}{\gamma}-\tfrac{p-1}{\tilde r}.$

If $p<1$, we define $\gamma_1$ so that $\tfrac{1}{\gamma_1}  = (p-\tfrac{\eps}{|s_c|})\tfrac{1}{\tilde r}$.  For $\eps>0$ small enough, using the fractional chain rule for H\"older continuous functions and \eqref{scaling}, and estimating as above, we find
\begin{align*}
\| |\nabla|^{\eps}G_k(\tilde u_\theta) \|_{L_t^{\frac{\tilde q}{p},\infty} L_x^{\gamma}} & \lesssim \bigl\| \,\||\tilde u_\theta|^{p-\frac{\eps}{|s_c|}}\|_{L_x^{\gamma_1}} \||\nabla|^{|s_c|}\tilde u_\theta\|_{L_x^{r_2}}^{\frac{\eps}{|s_c|}}\bigr\|_{L_t^{\frac{\tilde q}{p},\infty}}\lesssim \|u_\theta\|_W^p.
\end{align*}
This completes the proof.
\end{proof}

\subsection{Local well-posedness and stability}  In this subsection, we review the local theory for \eqref{nls} under the assumption \eqref{asmp:p}.  The results presented are consequences of Strichartz (Proposition~\ref{prop:Strichartz}), together with the estimates of the previous subsection (Lemma~\ref{L:embeddings} and Lemma~\ref{L:nonlinear}).  For details, see \cite{NakOza, Mas2}.  

The first result gives local existence and uniqueness in $C_t \dot X^{|s_c|}\cap W$. 

\begin{theorem}[Local well-posedness] There exists $\delta>0$ such if $t_0\in\R$ and $u_0\in\dot X^{|s_c|}(t_0)$ satisfies 
\[
\|e^{i(t-t_0)\Delta}u_0\|_{W(I)}\leq\delta,
\] then there exists a unique solution $u\in C(I;\dot X^{|s_c|})\cap W(I)$ to \eqref{nls} with $u(t_0)=u_0$ and
\[
\|u\|_{W(I)}\lesssim\|e^{i(t-t_0)\Delta}u_0\|_{W(I)}.
\]

Similarly, if $u_0\in\fhsc$ and $I=(-\infty,T)$ is an interval such that 
\[
\|e^{it\Delta}u_0\|_{W(I)}\leq \delta,
\] 
then there exists a unique solution $u\in C(I;\dot X^{|s_c|})\cap W(I)$ to \eqref{nls} with data $u_0$ at $t_0=-\infty$ satisfying $\|u\|_{W(I)}\lesssim\|e^{it\Delta}u_0\|_{W(I)}.$
\end{theorem}

We also have the following blowup/scattering criterion in terms of the $S$-norm. 

\begin{proposition}[Blowup/scattering criterion]
Let $u:I_{\max}\times\R^d\to\C$ be a maximal-lifespan solution solution to \eqref{nls}.  Suppose that $I_{\max}\ni 1$. 
\begin{itemize}
\item If $T_{\min}>-\infty$, then $\lim_{t\downarrow T_{\min}}S_{(t,1]}(u)=\infty$.  
\item If $T_{\max}<\infty$, then $\lim_{t\uparrow T_{\max}}S_{[1,t)}(u)=\infty$. 
\item If $T_{\max}=\infty$ and $S_{(1,\infty)}(u)<\infty$, then $u$ scatters in $\fhsc$ forward in time. 
\end{itemize}
\end{proposition}

The last two results together give global well-posedness and scattering for sufficiently small data.

We next record a stability result from \cite{Mas2}, which plays an important role in the construction and analysis of minimal blowup solutions.

\begin{proposition}[Stability, \cite{Mas2}]\label{thm:lpt}  Let $I$ be an interval and $t_0\in I$.  Let $\tilde{u}:I\times\R^d\to\C$ satisfy
\[
(i\partial_t+\Delta)\tilde u = \mu |\tilde u|^p \tilde u + \tilde e\qtq{and}\|\tilde u\|_{W(I)}\leq M
\]
for some function $\tilde{e}$ and some $M>0$.  Let $u_0\in\dot X^{|s_c|}(t_0)$, and let $u$ be defined on an interval containing $t_0$ and satisfy
\[
(i\partial_t+\Delta) u = \mu |u|^p  u + e, \quad u(t_0)=u_0,
\]
where $e:I\times\R^d\to\C$.  There exists $\eps_1=\eps_1(M)$ such that if 
\[
\|\tilde u(t_0)-u_0\|_{\dot X^{|s_c|}(t_0)}+ \|e\|_{N(I)} + \|\tilde e\|_{N(I)} < \eps
\]
for some $0<\eps<\eps_1$, then $u$ is defined for all $t\in I$ and satisfies
\[
\|u-\tilde u\|_{L_t^\infty \dot X^{|s_c|}(I\times\R^d)\cap W(I)}\lesssim_{M} \eps^\beta\qtq{for some}\beta\in(0,1].
\] 
\end{proposition}

We typically apply Proposition~\ref{thm:lpt} with $\tilde e \neq 0$ and $e\equiv 0$, using an approximate solution to deduce information about a true solution.  In Lemma~\ref{lem:unJ}, we do the opposite: we apply Proposition~\ref{thm:lpt} with $\tilde e \equiv 0$ and $e\neq 0$, using a true solution to deduce information about an approximate solution.

We also record the following corollary.

\begin{corollary}\label{C:stability} Let $u:I_{\max}\times\R^d\to\C$ be a maximal-lifespan solution to \eqref{nls}, $t_0\in I_{\max}$, and $I\ni t_0$.  Suppose
\begin{equation}\label{C:stabasmp}
\| e^{-it_0\Delta} u(t_0) \|_{\fhsc} \leq M \qtq{and} \|e^{i(t-t_0)\Delta} u(t_0)\|_{S(I)} < \eps
\end{equation}
for some $M, \eps>0$. If $\eps$ is sufficiently small, then $I\subset I_{\max}$ and
\begin{equation}\label{C:stabconc}
\|u-e^{i(t-t_0)\Delta}u(t_0)\|_{L_t^\infty \dot X^{|s_c|}(I\times\R^d)\cap W(I)} \lesssim_M \eps^\beta\qtq{for some}\beta>0.
\end{equation}
\end{corollary}
\begin{proof} We will apply Proposition~\ref{thm:lpt} on $I$, taking $\tilde u(t) := e^{i(t-t_0)\Delta}u(t_0)$ as an approximate solution.  By Strichartz and \eqref{C:stabasmp}, we have
\[
\|\tilde u\|_{W(I)} \lesssim M.
\] 
Furthermore, by Lemma~\ref{L:nonlinear}, Lemma~\ref{L:embeddings}, and \eqref{C:stabasmp}, $\tilde u$ solves \eqref{nls} up to an error that is bounded by
\[
\|\, |\tilde u|^p \tilde u\|_{N(I)} \lesssim \|\tilde u\|_{S(I)}^p \|\tilde u\|_{W(I)} \lesssim M\eps^p. 
\] 
Moreover, $\tilde u(t_0)=u(t_0)$.  Taking $\eps$ sufficiently small, we can invoke Proposition~\ref{thm:lpt} and uniqueness to conclude that $I\subset I_{\max}$ and that the bound \eqref{C:stabconc} holds. 
\end{proof}

\subsection{Concentration compactness}  We recall the linear profile decomposition from \cite{Mas2}, which will be a crucial tool in the construction of minimal blowup solutions in Section~\ref{S:existence}. 

\begin{proposition}[Linear profile decomposition, \cite{Mas2}]\label{prop:pd}  Let $\{\phi_n\}$ be a bounded sequence in $\fhsc$. Passing to a subsequence, there exist non-zero profiles $\psi^j \in \fhsc$, parameters $\xi_n^j\in\R^d,$ $h_n^j\in(0,\infty)$, and remainders $W^J_n\in\fhsc$ such that
\begin{equation}\label{eq:pd1}
\phi_n = \sum_{j=1}^J e^{ix\xi_n^j} \psi^j_{\{ h_n^j \}} + W^J_n\qtq{for all}J\geq 1.
\end{equation}

For each $J\geq 1$, we have the decoupling
\begin{equation}\label{eq:pd3}
\norm{\phi_n}_{\fhsc}^2 = \sum_{j=1}^J \norm{\psi^j}_{\fhsc}^2  + \norm{W_n^J}_{\fhsc}^2 + o(1)\qtq{as}n\to\infty.
\end{equation}

The parameters are asymptotically orthogonal: if $j\neq k$, then
\begin{equation}\label{eq:pd2}
\tfrac{h_n^j}{h_n^k} + \tfrac{h_n^k}{h_n^j} + \tfrac{|\xi^j_n-\xi^k_n|}{h_n^j} \to \infty\qtq{as}n\to\infty.
\end{equation}

Finally, the remainders satisfy
\begin{equation}\label{eq:pd5}
(e^{-ix\xi_n^j} W_n^J)_{\{\frac{1}{h_n^j}\}} \rightharpoonup 0 \qtq{weakly in}\fhsc\qtq{as}n\to\infty
\end{equation}
for all $1\leq j\leq J$ and vanish in Strichartz norms:
\begin{equation}\label{eq:pd4}
\liminf_{J\to\infty} \limsup_{n\to\infty} \norm{e^{it\Delta}W^J_n}_{L_t^{q,\infty}L_x^r(\R\times\R^d)} =0
\end{equation}
 for any $1<q,r<\infty$ such that $\tfrac{1}{q}\in(s_c,s_c+\tfrac12)$ and $\tfrac{2}{q}+\tfrac{d}{r}=\tfrac{2}{p}.$
\end{proposition}

\section{Existence of minimal blowup solutions}\label{S:existence}

In this section, we prove that if Theorem~\ref{T:main} fails, then we can construct minimal blowup solutions with a good compactness property, namely, almost periodicity modulo the symmetries of the equation; see Theorem~\ref{thm:apsol}.  For an introduction to the techniques we will be using, we refer the reader to \cite{KV1, Vis3}.  

In this paper, we consider for the first time almost periodic solutions in the mass-subcritical setting.  Compared to the case of non-negative critical regularity, we prove that it is the orbit of the interaction variable $f(t)=e^{-it\Delta}u(t)$ that is almost periodic, rather than the orbit of the solution $u$ itself.  Because of this, it turns out to be prudent not to attack Theorem~\ref{T:main} directly, but to first recast it in a form more amenable to analysis; this is the role of Theorem~\ref{P:main}.

\begin{definition}[Almost periodic]  Let $u:I\times\R^d\to\C$  and define $f(t):=e^{-it\Delta}u(t)$.  We say $u$ is \emph{almost periodic} (modulo symmetries) on $I$ if 
\[
f \in L_t^\infty \fhsc(I\times\R^d)
\] 
and there exist $h:I \to (0,\infty)$, $\xi:I \to \R^d$, and $C:(0,\infty) \to (0,\infty)$ such that
\begin{equation}\label{eq:ap1}
\int_{|x|\le \frac{C(\eta)}{h(t)}} \bigl||x|^{|s_c|}f(t)\bigr|^2\, dx + \int_{|\xi - \xi(t)| \ge {C(\eta)}{h(t)}} \bigl||\nabla|^{|s_c|} \widehat{f}(t)\bigr|^2 \,d\xi \le \eta\qtq{for all}t\in I.
\end{equation}
We call $\xi(t)$ the \emph{frequency center}, $h(t)$ the \emph{frequency scale}, and $C(\eta)$ the \emph{compactness modulus}. 
\end{definition}

\begin{remark}\label{rmk:ap} By Arzel\`a--Ascoli, $u$ is almost periodic if and only if
\[
K:= \bigl\{ \bigl(e^{-ix\xi(t)} e^{-it \Delta} u(t)\bigr)_{\{\frac{1}{h(t)}\}}: t \in I \bigr\}
\]
is pre-compact in $\fhsc$. This also implies that
\begin{equation}\label{eq:ap3}
e^{-it\Delta} u(t,x) = h(t)^{\frac{2}{p}} e^{ix\xi(t)} \psi(t,h(t) x) \qtq{for some} \psi:I \to K,
\end{equation}
that is, the orbit of $f(t)=e^{-it\Delta}u(t)$ is pre-compact in $\fhsc$ modulo scaling and Galilei boosts.
\end{remark}

The goal of this section is to establish the following. 

\begin{theorem}[Reduction to almost periodic solutions]\label{thm:apsol}  Suppose Theorem~\ref{T:main} fails.  Then there exists a maximal-lifespan solution $u$ to \eqref{nls} with $1\in I_{\max}$ such that
\begin{itemize}
\item[(i)] $u$ does not scatter forward in time,
\item[(ii)] $u$ is almost periodic on $[1,T_{\max})$. 
\end{itemize}
Furthermore, the frequency scale satisfies $h(t)\lesssim_u t^{-\frac12}$. 
\end{theorem}

Note that Theorem~\ref{thm:apsol} makes no reference to the initial time $t_0$ appearing in Theorem~\ref{T:main}, despite the broken time-translation symmetry (cf. Remark~\ref{R:tt}) of the problem.  As a substitute for time translation, we will make use of the scaling symmetry \eqref{eq:scale1.5} to reduce Theorem~\ref{T:main} to a well-posedness statement about the case $t_0=1$.   The following property plays an important role:
\begin{equation}\label{eq:scale2}
f \in e^{it \Delta} \fhsc \Longleftrightarrow f_{\{\lambda\}} \in e^{it\lambda^{-2}\Delta} \fhsc.
\end{equation}

Consider the case $t_0\in(0,\infty)$.  By scaling, prescribing data $u_0\in e^{it_0\Delta}\fhsc$ at $t=t_0$ is equivalent to prescribing  data $u_{0\{\sqrt{t_0}\}}\in e^{i\Delta}\fhsc$ at $t=1$.  Furthermore, rescaling has no effect on the \emph{a priori} bound \eqref{asmp:bdd}. Thus, to prove Theorem~\ref{T:main} for arbitrary $t_0\in(0,\infty)$, it suffices to prove it when $t_0=1$.  Similarly, the case $t_0\in(-\infty,0)$ can be reduced to the case $t_0=-1$.  Thus, it suffices to consider only $t_0\in\{-\infty,-1,0,1\}$. 

Next, suppose we prescribe data $u_0\in\fhsc$ at $t=0$.  By local well-posedness, a solution $u$ exists on some interval $(-\delta,\delta)$.  The rescaling $u_{[\sqrt{\delta/2}]}$ is now a solution on $(-2,2)$ with $u_{[\sqrt{\delta/2}]}(1)\in e^{i\Delta}\fhsc$.  Thus, if we can prove Theorem~\ref{T:main} with $t_0=1$, we can also handle $t_0=0$.  A similar argument reduces the case $t_0=-\infty$ to the case $t_0=-1$.  

The arguments above also show that the case $t_0=-1$ reduces to the case $t_0=1$, provided we prove that \eqref{asmp:bdd} implies $T_{\max}>0$.  By the time-reversal symmetry $u(t,x)\mapsto \overline{u(-t,x)}$, this is equivalent to proving that \eqref{asmp:bdd} implies $T_{\min}<0$ when $t_0=1$.  

In conclusion, we have reduced Theorem~\ref{T:main} to the following:

\begin{theorem}\label{P:main} Let $u_0\in e^{i\Delta}\fhsc$. Let $u:I_{\max}\times\R^d\to\C$ be the maximal-lifespan solution to \eqref{nls} with $e^{-i\Delta}u(1)=u_0$. Suppose
\[
\sup_{t\in I_{\max}} \|e^{-it\Delta}u(t)\|_{\fhsc} <\infty.
\]
Then $T_{\min}<0$, $T_{\max}=\infty$, and $u$ scatters in $\fhsc$ forward in time.
\end{theorem}
 
It remains to prove Theorem~\ref{P:main}.  To this end, we define the following two quantites:
\begin{equation}\label{def:Einf}
\E_\infty := \inf\bigl\{ \limsup_{t\uparrow T_{\max}} \| e^{-it\Delta}u(t)\|_{\fhsc}\bigr\},
\end{equation}
where the infimum is taken over all maximal-lifespan solutions to \eqref{nls} with $1\in I_{\max}$ that do \emph{not} scatter forward in time, and
\begin{equation}\label{def:E0}
\E_0 := \inf\bigl\{ \limsup_{t\downarrow T_{\min}} \|e^{-it\Delta}u(t)\|_{\fhsc}\bigr\},
\end{equation}
where the infimum is taken over all maximal-lifespan solutions to \eqref{nls} with $1\in I_{\max}$ such that $T_{\min}\geq 0$.  To prove Theorem~\ref{P:main}, we must show $\E_0=\E_\infty=\infty$. 

\begin{remark}  At first glance, the statement $T_{\min}<0$ in Theorem~\ref{P:main} might appear easier to prove than the forward scattering conclusion of this theorem.  However, if $\E_\infty<\infty$, then we can construct a solution that blows up at $t=0$ (see Theorem~\ref{thm:ss} below), giving $\E_0<\infty$.  That is, $\E_0=\infty$ implies $\E_\infty=\infty$.  In fact, one can also show that $\E_\infty=\infty$ implies $\E_0=\infty$.  Thus, proving that the solution in Theorem~\ref{P:main} extends to negative times is as hard as proving forward scattering for this solution.
\end{remark}

In order to treat these two problems simultaneously, we introduce the quantity $L:[0,\infty)\to[0,\infty]$, defined by
\[
L(E) = \sup\{S_{I}(u)\},
\]
where the supremum is taken over all solutions to \eqref{nls} on a compact interval $I \subset (0,\infty)$
with $1\in I$ and
\[
\sup_{t\in I} \|e^{-it\Delta}u(t)\|_{\fhsc} \leq E.
\]
We now define
\begin{equation}\label{def:Ec}
\E_c := \inf \{ E : L(E) = \infty \} = \sup \{ E: L(E)<\infty \}.
\end{equation}

An important relationship between the quantities just defined is the following. 

\begin{lemma}\label{lem:E0} If $\min (\E_0,\E_\infty) <\infty$, then $\E_c<\infty$.
\end{lemma} 
\begin{proof}  We will show that if $\min(\E_0,\E_\infty)<\infty$, then $\E_c\leq \min(\E_0,\E_\infty)$. 

First, suppose $\E_\infty<\infty$.  By definition, for any $\eps>0$ there exists a solution $ u$ such that  $1\in I_{\max}$, $ u$ does not scatter forward in time, and
\[
\limsup_{t \uparrow T_{\max}} \cn{e^{-it\Delta}{u}(t)} \le \E_\infty + \eps.
\]
In particular, there exists $0<t_0 \in I_{\max}$ such that
\[
E:=\sup_{t\in [t_0, T_{\max})} \cn{e^{-it\Delta}{u}(t)} \le \E_\infty + 2\eps.
\]

As ${u}$ does not scatter forward in time, we have $S_{[t_0, T_{\max})}({u}) =\infty.$  As the lifespan of the rescaled solution $u_{[\sqrt{t_0}]}$ contains $t=1$, we deduce that $L(E)=\infty$.  This implies that $\E_c\leq \E_\infty+2\eps$.  As $\eps$ was arbitrary, we conclude $\E_c\leq \E_\infty$. 

A similar argument shows that $\E_0<\infty$ implies $\E_c\leq \E_0$. 
\end{proof}

A direct consequence of Lemma~\ref{lem:E0} is the following

\begin{corollary}
If $\E_c=\infty$, then Theorem~\ref{T:main} holds.
\end{corollary}

We will prove that $\E_c=\infty$ by contradiction.  The first major step is Theorem~\ref{thm:apsol}, in which we will show that if $\E_c<\infty$, we can find almost periodic solutions.  The key ingredient in the proof of Theorem~\ref{thm:apsol} is the following convergence result for minimizing sequences. 

\begin{proposition}[Key convergence result]\label{prop:key}  Suppose $\E_c<\infty$.  Suppose that there exist intervals $I_n\subset(0,\infty)$ and solutions $u_n:I_n\times\R^d\to\C$ to \eqref{nls} such that 
\begin{equation}\label{asmp:PS1}
\limsup_{n\to\infty} \sup_{t\in I_n} \norm{ e^{-it\Delta} u_n(t)}_{\fhsc} = \E_c.
\end{equation}
Suppose further that there exist $t_n\in I_n$ satisfying
\begin{equation}\label{asmp:PS2}
\lim_{n\to\infty} S_{I_n \cap [ t_n,\infty) } (u_n) =\infty.
\end{equation}
Then $\E_\infty=\E_c$, and the following hold: 
\begin{itemize}
\item[(i)] There exist non-zero $\psi \in\fhsc$, $\tau\in[0,\infty)$, $\xi_n \in \R^d$, and $h_n\in(0,\infty)$ such that 
\[
h_n^2 t_n \to \tau\qtq{and} \bigl(e^{-ix\xi_n} e^{-it_n\Delta} u_n (t_n)\bigr)_{\{\frac{1}{h_n}\}} \to  \psi\qtq{in}\fhsc
\]
as $n\to\infty$ along a subsequence. 
\item[(ii)] The solution $u$ to \eqref{nls} with $e^{-i\tau\Delta}u(\tau)=\psi$ does not scatter forward in time and
\begin{equation}\label{sat}
\sup_{ t\in [\tau,T_{\max})} \cn{e^{-it\Delta} u(t)} = \limsup_{ t\uparrow T_{\max}} \cn{e^{-it\Delta} u(t)} = \E_c.
\end{equation}
\end{itemize}
\end{proposition}

\begin{proof} We apply Proposition~\ref{prop:pd} to write
\begin{equation}\label{eq:key_proof1}
e^{-it_n\Delta}u_n(t_n) = \sum_{j=1}^J e^{ix\xi_n^j} \psi^j_{\{ h_n^j \}}+ W^J_n
\end{equation}
along a subsequence. We modify the decomposition in the following way: for any $j$ such that $(h_n^j)^2t_n\to\infty$, we add the profile $e^{ix\xi_n^j}\psi^j_{\{h_n^j\}}$ to the remainder term $W_n^J$.  For these profiles, rescaling, Strichartz, and the monotone convergence theorem imply 
\[
\norm{e^{it\Delta}e^{ix\xi_n^j} \psi^j_{\{ h_n^j \}}}_{L_t^{q,\infty}L_x^r([t_n,\infty)\times\R^d)}\to0\qtq{as}n\to\infty.
\]
Note that the vanishing of the remainder $W_n^J$ now holds only in the weaker sense 
\begin{equation}\label{eq:mod_small}
	\liminf_{J\to\infty} \limsup_{n\to\infty} \norm{e^{it\Delta}W^J_n}_{L_t^{q,\infty}L_x^r([t_n,\infty)\times\R^d)} =0,
\end{equation}
rather than the original sense \eqref{eq:pd4}. 

Having made this modification, we can assume that each $(h_n^j)^2 t_n$ is bounded in $n$, so that $(h_n^j)^2 t_n\to \tau^j\in[0,\infty)$ along a subsequence.

Now let $\Psi^j: I^j \times \R^d \to \C$ be the maximal-lifespan solution to \eqref{nls} with $ e^{-i\tau^j\Delta}\Psi^j(\tau^j) = \psi^j$ and define
\[
v_n^j(t,x) := e^{ix\xi_n^j} e^{-it|\xi_n^j|^2} \Psi^j_{[h_n^j]}(t,x-2t\xi_n^j),
\]
which solves \eqref{nls} and satisfies
\begin{equation}\label{eq:v_Psi}
e^{-it\Delta} v_n^j(t) = e^{ix\xi_n^j} e^{-it\Delta} \Psi^j_{[h_n^j]}(t) = e^{ix\xi_n^j}\bigl (e^{-i(h_n^j)^2 t \Delta} \Psi^j ((h_n^j)^2t)\bigr)_{\{ h_n^j\}}.
\end{equation}

We also define the functions
\[
{u}^J_n(t) := \sum_{j=1}^J v_n^j(t) + e^{it\Delta} W_n^J,
\]
which we regard as approximate solutions to \eqref{nls} that asymptotically match $u_n$ at time $t_n$.  Indeed, such asymptotic agreement at time $t_n$ follows by construction:
\begin{align*}
e^{-it_n\Delta}({u}_n^J(t_n) - u_{n}(t_n))
&= \sum_{j=1}^J e^{-it_n\Delta} v_n^j(t_n) - e^{ix\xi_n^j} \psi^j_{\{ h_n^j \}} \\
&= \sum_{j=1}^J e^{ix\xi_n^j} \bigl[(e^{-i(h_n^j)^2t_n\Delta} \Psi^j((h_n^j)^2t_n))_{\{h_n^j\}}-\psi^j_{\{ h_n^j \}} \bigr],
\end{align*}
and hence, by the definition of $\Psi^j$ we have
\begin{equation}\label{E:unJ1}
\norm{e^{-it_n\Delta}({u}_n^J(t_n) - u_{n}(t_n))}_{\fhsc}\to 0\qtq{as}n\to\infty.
\end{equation}

To see that $u_n^J$ are actually approximate solutions will require control over the functions $v_n^j$.  As we will see below, the decoupling \eqref{eq:pd3} and the boundedness \eqref{asmp:PS1} imply good control for large $j$.  The following lemma shows that whenever we have control over \emph{all} of the $v_n^j$, then $u_n^J$ also obey good bounds and approximately solve \eqref{nls}.

\begin{lemma}\label{lem:unJ} Suppose $\tilde I_n=(t_n,t_n+T_n)$ are intervals such that
\begin{equation}\label{goodcontrol}
\sup_j\limsup_{n\to\infty} \|v_n^j\|_{W(\tilde I_n)} \lesssim 1.
\end{equation}
Then the functions $u_n^J$ satisfy the following:
\begin{itemize}
\item[(i)] Inheritance of bounds:
\[
\limsup_{J\to\infty}\limsup_{n\to\infty} \| u_n^J\|_{W(\tilde I_n)} \lesssim 1. 
\]
\item[(ii)] Approximate solutions: Writing $F(u)=\mu |u|^p u$, 
\[
\limsup_{J\to\infty}\limsup_{n\to\infty} \|(i\partial_t+\Delta)u_n^J - F(u_n^J) \|_{N(\tilde I_n)} = 0. 
\]
\end{itemize}
\end{lemma}

\begin{proof}  We only sketch the proof; for complete details, one can refer to \cite{Mas2}. 

We begin with (i). Let $J_0\geq1$ to be determined below.  For $J>J_0$, define
\[
z_{n,J_0}^J(t) := \sum_{j=J_0+1}^J v_n^j(t).
\] 
Using the decoupling \eqref{eq:pd3} and boundedness \eqref{asmp:PS1}, for any $\eps>0$ we may find $J_0=J_0(\eps)$ large enough that
\[
\limsup_{n\to\infty} \norm{ e^{-it_n\Delta}z_{n,J_0}^J(t_n)}_{\fhsc}^2 \le\sum_{j=J_0+1}^J \|\psi^j\|_{\fhsc}^2 <\eps^2.
\]
For example, we can choose $\eps$ to be half of the small-data threshold.  In particular, for $n$ large, the solution $\tilde u_n$ to \eqref{nls} with $\tilde u_n(t_n)=z_{n,J_0}^J(t_n)$ is global and satisfies
\[
\|\tilde u_n\|_{W(\R)} \lesssim \eps.
\]
The orthogonality \eqref{eq:pd2} and the bounds \eqref{goodcontrol} imply that $z_{n,J_0}^J$ asymptotically solve \eqref{nls}.  Therefore, invoking Proposition~\ref{thm:lpt}, we deduce
\[
\limsup_{n\to\infty}\norm{z_{n,J_0}^J}_{W(\tilde I_n)} \lesssim \eps^\beta\qtq{for all}J>J_0
\]
(see \cite[Lemma~7.1]{Mas2} for more details).  Invoking \eqref{goodcontrol} for $j<J_0$, we obtain (i).

We turn to (ii). Let $\eps>0$ and choose $J_0=J_0(\eps)$ as above.  We now write
\begin{align}
\|&(i\partial_t+\Delta)u_n^J - F(u_n^J) \|_{N(\tilde I_n)} \nonumber \\
&\lesssim \|F(u_n^J) - F(u_n^J - z_{n,J_0}^J) \|_{N(\tilde I_n)} + \bigl\| F\bigl(\sum_{j=1}^{J_0} v_n^j\bigr) - F\bigl(\sum_{j=1}^J v_n^j\bigr)\bigr\|_{N(\tilde I_n)} \label{err1} \\
& \quad +\bigl\|F\bigl(\sum_{j=1}^J v_n^j\bigr) - \sum_{j=1}^J F(v_n^j) \bigr\|_{N(\tilde I_n)} + \bigl\| F(u_n^J - z_{n,J_0}^J) - F\bigl(\sum_{j=1}^{J_0} v_n^j\bigr) \bigr\|_{N(\tilde I_n)}. \label{err2}
\end{align}

As 
\begin{align*}
\eqref{err1} \lesssim  \|z_{n,J_0}^J\|_{W(\tilde I_n)}\bigl[ \|u_n^J\|_{W(\tilde I_n)} + \|\textstyle\sum_{j=1}^{J_0} v_n^j\|_{W(\tilde I_n)} + \|W_n^J\|_{\fhsc}\bigr]^p,
\end{align*} 
we can use the boundedness of $u_n^J$ and smallness of $z_{n,J_0}^J$ to deduce
\[
\limsup_{J\to\infty}\limsup_{n\to\infty} \eqref{err1} \lesssim \eps^\beta.
\]

Exploiting the orthogonality \eqref{eq:pd2} again, one can show that the first term in \eqref{err2} is $o(1)$ as $n\to\infty$ for each $J$.  

Finally, writing
\[
u_n^J - z_{n,J_0}^J = \sum_{j=1}^{J_0} v_n^j + e^{it\Delta} W_n^J,
\]
we can use the vanishing condition \eqref{eq:mod_small} to show that the second term in \eqref{err2} is $o(1)$ as $n\to\infty$ and then $J\to\infty$.  The complete details may be found in \cite[Lemma~7.2]{Mas2}.
\end{proof}

Our first goal is to show that there exists \emph{at least one} non-scattering profile. 

\begin{lemma}\label{lem:key_proof1}
There exists $j$ such that $\Psi^j$ does not scatter forward in time.
\end{lemma}
\begin{proof} Assume towards a contradiction that every $\Psi^j$ scatters forward in time.  Then $I_j\supset[\tau^j,\infty)$ and  $\norm{\Psi^j}_{W(I_j)} <\infty$.

For large $n$,
\[
\norm{v_n^j}_{W([t_n , \infty))}= \norm{\Psi^j_{[h_n^j]}}_{W([t_n , \infty))} \le \norm{\Psi^j}_{W(I_j)} <\infty.
\]
Thus, we can apply Lemma~\ref{lem:unJ} on the intervals $[t_n,\infty)$ to see that the $u_n^J$ are approximate solutions to \eqref{nls} for $n,J$ large.  Furthermore, by \eqref{E:unJ1}, $u_n^J$ asymptotically match $u_n$ at time $t_n$.   Applying Proposition~\ref{thm:lpt}, we deduce that
\[
\norm{u_n}_{W([t_n,\infty))} \lesssim 1
\]
for $n$ large, contradicting \eqref{asmp:PS2}.
\end{proof}

We now wish to prove that there is \emph{only one} non-scattering profile. We re-order indices so that $\Psi^j$ blows up if and only if $1\leq j\leq J_1$.  The decoupling \eqref{eq:pd3} and small-data theory guarantee that $J_1$ is finite.  We now wish to use minimality to prove that $J_1=1$.  To make this precise, we need to prove that the decoupling \eqref{eq:pd3} persists in time.  This is not obvious, as  the $\fhsc$-norm is not conserved.  We follow the arguments in \cite{KV2}.

For $m,n \ge 1$, define $j(m,n) \in \{1,2,\dots,J_1\}$ and $K_n^m=[t_n,t_n+T_n^m]$ by
\[
\sup_{1\le j \le J_1} \norm{v_n^j}_{W(K_n^m)} = \norm{v_n^{j(m,n)}}_{W(K_n^m)}  = m.
\]
By the pigeonhole principle, there exists $j_1 \in \{1,2,\dots,J_1\}$ so that for infinitely many $m$, one has $j(m,n) = j_1$ for infinitely many $n$.  By reordering the indices, we may assume that $j_1=1$.
By the definition of $\E_\infty$ and Lemma~\ref{lem:E0}, it follows that
\begin{equation}\label{eq:1stprofile}
	\limsup_{m\to\infty} \limsup_{n\to\infty} \sup_{t\in K_n^m} \norm{e^{-it\Delta}v_n^1(t)}_{\fhsc}
	\ge \E_\infty \ge \E_c.
\end{equation}

We are now in a position to prove that there is only one non-scattering profile.
\begin{lemma}\label{lem:key_proof2} The following hold:
\begin{itemize}
\item[(i)] $\psi^j\equiv 0$ for $j \ge 2$, 
\item[(ii)] $W_n^1 \to 0$ in $\fhsc$ as $n\to\infty$.
\end{itemize}
\end{lemma}
\begin{proof} By the definition of $K_n^m$,
\[
\sup_{n}\norm{v_n^j}_{W(K_n^m)} \le m \qtq{for}1\leq j\leq J_1.
\]

For $j>J_1$, $\Psi^j$ scatters forward in time and so we have 
\[
\norm{v_n^j}_{W([t_n,\infty))} \le \norm{\Psi^j}_{W(I_j)} < \infty.
\]
Applying Lemma~\ref{lem:unJ} and appealing to Proposition~\ref{thm:lpt}, it follows that
\begin{equation}\label{eq:approximation}
\limsup_{J\to\infty} \limsup_{n\to\infty}\sup_{t\in K_n^m}\norm{e^{-it\Delta}(u_n(t) - {u}_n^J(t))}_{\fhsc} = 0
\end{equation}
for each $m$. 

Now define
\[
c_j := \inf_{t \in I^j} \norm{e^{-it\Delta}\Psi^j(t)}_{\fhsc}
\]
for each $j$. We will show that $c_j=0$ for $j\ge2$, which implies that $\psi_j=0$ for $j\geq 2$, thus settling (i).

Suppose towards a contradiction that $c_{j_0} >0$ for some $j_0\ge2$.

Fix $\eps>0$. Choose $m=m(\eps)$ and an $m$-dependent subsequence in $n$ such that
\[
\sup_{t \in K_n^m} \norm{ e^{-it\Delta} v_n^1(t)}_{\fhsc}^2 \ge \E_c^2 -\eps.
\]
Next, using \eqref{eq:approximation}, choose $J=J(\eps)$ so that
\begin{align}\label{318}
\sup_{t\in K^m_n}\norm{e^{-it\Delta}(u_n(t) - {u}_n^J(t))}_{\fhsc}^2\le \eps
\end{align}
for all $n$ sufficiently large depending on $J$.  Without loss of generality, we may assume that $J>j_0$. 

Next, choose a sequence $\tilde t_n \in K_n^m$ so that 
\[
\norm{e^{-i\tilde t_n\Delta}v^1_n(\tilde t_n)}_{\fhsc}^2 \ge \sup_{t\in K_n^m} \norm{e^{-it\Delta}v^1_n(t)}_{\fhsc}^2 -\eps\geq \E_c^2-2\eps.
\]
We now claim that 
\begin{align}
\norm{e^{-i\tilde t_n\Delta}{u}_n^J(\tilde t_n)}_{\fhsc}^2  &= \sum_{j=1}^J \norm{e^{-i\tilde t_n\Delta}v^j_n(\tilde t_n)}_{\fhsc}^2 + \norm{W^J_n}_{\fhsc}^2  + o(1) \nonumber\\
&\geq  \sup_{t\in K_n^m} \norm{e^{-it\Delta}v^1_n(t)}_{\fhsc}^2 -\eps + c_{j_0}^2 + o(1)\label{later decouple}
\end{align}
as $n\to\infty$.  To see this, we need to show that
\begin{equation}
\langle e^{-i\tilde t_n\Delta} v_n^j(\tilde t_n), e^{-i\tilde t_n\Delta} v_n^k(\tilde t_n)\rangle + \langle  W_n^J, e^{-i\tilde t_n\Delta} v_n^j(\tilde t_n)\rangle \to 0 \qtq{as}n\to\infty  \label{cross1}
\end{equation}
whenever $1\leq j,k\leq J$ and $j\neq k$. Here $\langle f,g\rangle = \int |x|^{2|s_c|} f\bar g\,dx.$  The proof of \eqref{cross1} follows the argument in \cite{KV2}:  

The definition of $K_n^m$ guarantees that each $(h_n^j)^2 \tilde t_n$ belongs to a closed interval inside $I_j$ for all $1\leq j\leq J_1$.  Using also the fact that for $j>J_1$, the solutions $\Psi^j$ obey uniform bounds, we can find $f^j$ such that
\[
e^{-i(h_n^j)^2\tilde t_n\Delta} \Psi^j((h_n^j)^2 \tilde t_n) \to f^j\qtq{strongly in}\fhsc
\]
along a subsequence.  Using \eqref{eq:v_Psi}, we can replace $e^{-i\tilde t_n\Delta}v_n^j(\tilde t_n)$ by $e^{ix\xi_n^j}f^j_{\{h_n^j\}}$.  Thus \eqref{cross1} follows from the orthogonality \eqref{eq:pd2} and the weak decoupling \eqref{eq:pd5}.  

Continuing from \eqref{later decouple} and using the definition of $m$,  
\[
\sup_{t\in K_n^m} \norm{e^{-it\Delta}{u}_n^J(t)}_{\fhsc}^2 \ge \E_c^2 + c_{j_0}^2 - 3\eps
\]
for large $n$.  On the other hand, by assumption \eqref{asmp:PS1} and \eqref{318}, we have
\[
\sup_{t\in K_n^m} \norm{e^{-it\Delta}{u}_n^J(t)}_{\fhsc}^2 \le \E_c^2 + \eps
\]
for large $n$.  We deduce that
\[
\E_c^2 + \eps \geq \E_c^2 + c_{j_0}^2 - 3\eps.
\]
Choosing $\eps$ sufficiently small now contradicts $c_{j_0}>0$, which completes the proof of (i).  

Consequently, we have $W_n^J\equiv W_n^1$.  Arguing as above, we can similarly deduce that
\[
\limsup_{n\to\infty}\norm{W_n^1}_{\fhsc}^2 \le \eps
\] 
for any $\eps>0$, thus proving (ii).
\end{proof}

Returning to the decomposition \eqref{eq:key_proof1}, we have obtained part (i) of Proposition~\ref{prop:key} with $\psi:=\psi^1$, $\xi_n:=\xi_n^1$, and $h_n:=h_n^1$.

For part (ii), we take $u:=\Psi^1$, which we have already proven does not scatter forward in time.  Furthermore, arguing as in Lemma~\ref{lem:key_proof2}, 
\[
\sup_{t\in [\tau,T_{\max})} \norm{e^{-it\Delta}u(t)}_{\fhsc} \le \E_c.
\]
This implies that $\E_\infty \le \E_c <\infty$.  In light of Lemma~\ref{lem:E0}, we conclude that $\E_\infty=\E_c$ and that \eqref{sat} holds.  This completes the proof of Proposition~\ref{prop:key}. 
\end{proof}

We turn to the proof of Theorem~\ref{thm:apsol}.
\begin{proof}[Proof of Theorem~\ref{thm:apsol}]  We suppose that Theorem~\ref{T:main} fails, so that $\E_c<\infty$.

We take a sequence $\eps_n\to 0$ such that $L(\E_c-\eps_n)\geq n$.  In particular, there exist closed intervals $I_n\subset(0,\infty)$ and solutions $u_n$ on $I_n$ so that
\[
S_{I_n}(u_n)\geq n-1 \qtq{and} \sup_{t\in I_n} \cn{e^{-it\Delta}u_n(t)}\leq \E_c-\eps_n.
\]
We set $t_n =\min I_n$.  Note also that 
\[
\limsup_{n\to\infty} \sup_{t\in I_n}\|e^{-it\Delta}u_n(t)\|_{\fhsc} = \E_c,
\]
for otherwise $S_{I_n}(u_n)$ would remain bounded (cf. the definition of $\E_c$). 

This shows that the assumptions of Proposition~\ref{prop:key} are satisfied for the intervals $I_n$, the solutions $u_n$, and the times $t_n$.  Thus, $\E_c=\E_\infty$ and we can find a maximal-lifespan solution $u$ with $[0,\infty)\cap I_{\max}\neq \emptyset$, which does not scatter forward in time and satisfies
\begin{equation}\label{uisminimal}
\sup_{[\tau,T_{\max})} \|e^{-it\Delta}u(t)\|_{\fhsc} = \limsup_{t\uparrow T_{\max}} \|e^{-it\Delta}u(t)\|_{\fhsc} = \E_\infty
\end{equation}
for some $\tau\geq 0$. By scaling, we may assume that $I_{\max}\ni 1$.

To see that $u$ satisfies the conclusions of Theorem~\ref{thm:apsol}, it remains only to prove that $u$ is almost periodic on $[1,T_{\max})$, with $h(t)\lesssim_u t^{-\frac12}$.  Given a sequence $\{t_n\}\subset [1,T_{\max})$, the hypotheses of Proposition~\ref{prop:key} are satisfied with $u_n\equiv u$ and $I_n=[1,T_{\max})$, and so we deduce that $e^{-it_n\Delta}u(t_n)$ converges along a subsequence in $\fhsc$, modulo scaling and Galilei boosts.  This proves that $u$ is almost periodic on $[1,T_{\max})$.

To see that the scaling parameter satisfies $h(t)\lesssim_u t^{-\frac12}$, we argue by contradiction:  Any sequence $\{t_n\}\subset [1,T_{\max})$ such that $t_n h(t_n)^2\to\infty$ as $n\to\infty$ would lead to a counterexample to Proposition~\ref{prop:key}(i) by taking $u_n$ to be $u$ restricted to the interval $[t_n,T_{\max})$. 
\end{proof}

\section{Reduction to self-similar solutions}\label{S:reduction}
In Theorem~\ref{thm:apsol}, we showed that the failure of Theorem~\ref{T:main} implies the existence of almost periodic solutions.  In this section, we show that the failure of Theorem~\ref{T:main} implies the existence of a special type of almost periodic solution, namely, a \emph{self-similar} almost periodic solution. 

\begin{definition}[Self-similar]\label{Def:SS} An almost periodic solution $u$ to \eqref{nls} is called \emph{self-similar} if $I_{\max}=(0,\infty)$, $\xi(t)\equiv 0$, and $h(t)=t^{-\frac12}$. 
\end{definition}

We warn the reader that this is a broader notion than that commonly encountered.   By Remark~\ref{rmk:ap}, a solution $u$ is self-similar in our sense if and only if
\begin{equation}\label{ssK}
e^{-it\Delta}u(t) = t^{-\frac{1}{p}} \psi(t, \tfrac{x}{\sqrt{t}}),
\end{equation}
where $\psi(t)$ takes values in a pre-compact set $K\subset\fhsc$.  The traditional notion of self-similar solution corresponds to the case where the set $K$ consists of a single point.

\begin{theorem}[Reduction to self-similar solutions]\label{thm:ss}  If Theorem~\ref{T:main} fails, then there exists a self-similar almost periodic solution.
\end{theorem}

We first prove a local constancy property for the modulation parameters of almost periodic solutions. 

\begin{proposition}[Local constancy]\label{prop:lc}
Let $u$ be a solution to \eqref{nls} with $1\in I_{\max}$.  Suppose $u$ is almost periodic on $[1,T_{\max})$. Then there exists $\delta=\delta(u)>0$ such that
\begin{equation}\label{Idelta}
\{t\in\R:|t-t_0|\leq \delta h(t_0)^{-2}\}\subset I_{\max}\qtq{for any} t_0\in[1,T_{\max}).
\end{equation}
Furthermore,
\[
h(t)\sim_u h(t_0)\qtq{and} \tfrac{|\xi(t)-\xi(t_0)|}{h(t_0)} \lesssim_u 1\qtq{whenever} |t-t_0|\leq\delta h(t_0)^{-2}. 
\]
\end{proposition}

\begin{proof} Fix $\eps>0$.  We claim that there exists $\delta=\delta(\eps,u)>0$ such that
\begin{equation}\label{eq:lc1}
\norm{e^{i(t-t_0)\Delta}u(t_0)}_{W([t_0-\delta h(t_0)^{-2},t_0+\delta h(t_0)^{-2}])} \le \eps\qtq{for any}t_0\in [1,T_{\max}).
\end{equation}

Using almost periodicity, specifically \eqref{eq:ap3}, we first write
\[
 e^{-it_0\Delta} u(t_0,x) = e^{ix\xi(t_0)} h(t_0)^{\frac2{p}} \psi (t_0,h(t_0) x), 
\]
where $\psi$ takes values in a pre-compact set $K\subset\fhsc$.  Next, using \eqref{gal},
\begin{align*}
\text{LHS\eqref{eq:lc1}} = \norm{e^{it\Delta}( &e^{ix\xi(t_0)} h(t_0)^{\frac2{p}} \psi (t_0,h(t_0) x)) }_{W([t_0-\delta h(t_0)^{-2},t_0+\delta h(t_0)^{-2}])} \\
&= \norm{e^{it\Delta} \psi(t_0)}_{W([t_0h(t_0)^2 -\delta, t_0h(t_0)^2 + \delta])}.
\end{align*}
Now recall that there exists $M=M(u)$ so that $0\leq t_0 h(t_0)^2 \leq M$. By compactness, there exists $\delta>0$ so that
\[
	\sup_{\phi \in K} \sup_{\sigma \in [0,M]} \norm{e^{it\Delta} \phi}_{W([\sigma -\delta, \sigma + \delta])} \le {\eps}.
\]
With this $\delta$, we have \eqref{eq:lc1}.

Choosing $\eps$ small, we can apply Corollary~\ref{C:stability} on the interval 
\[
[t_0-\delta h(t_0)^{-2},t_0+\delta h(t_0)^{-2}].
\]  This implies \eqref{Idelta} and shows that
\begin{equation}\label{eq:lc2}
\norm{e^{-it\Delta} u(t) - e^{-it_0\Delta}u(t_0)}_{\fhsc} \lesssim \eps^\beta \qtq{whenever}|t-t_0|\leq\delta h(t_0)^{-2}.
\end{equation}

We are now in a position to complete the proof of the proposition.  Suppose towards a contradiction that there exist sequences $\{t_n\}$ and $\{t_n' \}$
with $|t_n - t_n'| \le \delta h(t_n')^{-2}$ satisfying
\begin{equation}\label{eq:lc3}
	\bigl| \log \tfrac{h(t_n)}{h(t_n')} \bigr| + \bigl| \tfrac{\xi(t_n)-\xi(t_n')}{h(t_n')} \bigr| \to \infty\qtq{as}n\to\infty.
\end{equation}
By  \eqref{eq:lc2}, we have
\begin{equation}\label{const1}
\norm{e^{-it_n\Delta} u(t_n) - e^{-it_n'\Delta}u(t_n')}_{\fhsc} \lesssim\eps^\beta.
\end{equation}
On the other hand, a change of variables shows
\begin{align}
\| & e^{-it_n\Delta} u(t_n) - e^{-it_n'\Delta}u(t_n')\|_{\fhsc}^2 = \|\psi(t_n)\|_{\fhsc}^2 + \|\psi(t_n')\|_{\fhsc}^2 \label{eq:lc4}\\
&+  2\Re \int_{\R^d} e^{\frac{ix[\xi(t_n)-\xi(t_n')]}{h(t_n')}} \bigl(\tfrac{h(t_n)}{h(t_n')}\bigr)^{\frac{d}2}
	\bigl|\tfrac{h(t_n)}{h(t_n')} x\bigr|^{|s_c|} \psi\bigl(t_n,\tfrac{h(t_n)}{h(t_n')} x\bigr) |x|^{|s_c|}\bar\psi(t_n',x) \,dx. \notag
\end{align}
By the small data theory and the fact that $u$ does not scatter, we have that $e^{-it\Delta}u(t)$, and hence $\psi(t)$, remains bounded away from zero in $\fhsc$ uniformly in $t$.  The third term tends to zero as $n\to\infty$, as can be seen from \eqref{eq:lc3} and the fact that $\psi$ takes values in a compact set.  

Comparing \eqref{const1} and \eqref{eq:lc4}, we now reach a contradiction by choosing $\eps=\eps(u)$ sufficiently small.
\end{proof} 

\begin{corollary}\label{cor:ss} Suppose $u$ is a solution as in Theorem~\ref{thm:apsol}. Then the following hold: 
\begin{itemize}
\item[(i)] $u$ is forward global,
\item[(ii)] $h(t)\sim_u t^{-\frac12}$ for $t\geq 1$, 
\item[(iii)] There exists $\xi_0\in\R^d$ such that $|\xi(t)-\xi_0|\lesssim_u h(t)$ for $t\geq 1$. 
\end{itemize}
\end{corollary}

\begin{proof} Recall from Theorem~\ref{thm:apsol} that $h(t)\lesssim_u t^{-\frac12}$ for $t\geq 1$.  Thus item (i) follows immediately from \eqref{Idelta}.

For (ii), we argue by contradiction and suppose there is a sequence $t_n\geq 1$ satisfying 
\[
t_n h(t_n)^2\to 0\qtq{as} n\to\infty.
\]  
In particular, $h(t_n)\to 0$ as $n\to \infty$.  Choosing $\delta$ as in Proposition~\ref{prop:lc}, we have 
\[
|1-t_n| h(t_n)^2\leq \delta\qtq{for all large} n.
\]  
Proposition~\ref{prop:lc} now gives $h(t_n)\sim_u h(1)$ for all $n$ large, a contradiction. 

For (iii), we take $\delta$ as in Proposition~\ref{prop:lc} and fix $\eta>0$ so that
\[
\eta\cdot \sup_{t\geq 1} [t\cdot h(t)^2] \leq \delta.
\]
Define
\[
s_n=(1+\eta)^{n-1}\qtq{and} I_n=[s_n,s_{n+1}]\qtq{for}n\geq1.
\]
Then $[1,\infty)=\cup_{n\geq 1} I_n$, and by definition we have
\[
|t-s_n|\leq |s_{n+1}-s_n|=\eta s_n \leq \delta h(s_n)^{-2}\qtq{for any}n\geq1 \qtq{and}t\in I_n.
\]
By Proposition~\ref{prop:lc}, this gives
\begin{equation}\label{eq:ss1}
\sup_{t\in I_n}\tfrac{|\xi(t)-\xi(s_n)|}{h(s_n)} \lesssim_u 1\qtq{for all}n \geq 1.
\end{equation}  
In particular, as $h(s_n)\sim_u (1+\eta)^{-\frac{n-1}{2}}$, we have that
\[
	|\xi(s_n)-\xi(s_m)| \le \sum_{j=m}^{n-1} |\xi(s_{j+1})-\xi(s_j)|
	\lesssim_u \sum_{j=m}^{n-1} h(s_j)
	\lesssim_{u} h(s_m)\qtq{for}n>m.
\]
Thus $\{\xi(s_n)\}$ is a Cauchy sequence and so it converges to some $\xi_0\in\R^d$.

Now for any $t\geq 1$, there exists $I_n\ni t$, and we have
\[
|\xi(t)-\xi_0|\leq |\xi(t)-\xi(s_n)| + |\xi(s_n)-\xi_0|  \lesssim_u h(s_n) \sim_u h(t). 
\]
This completes the proof of Corollary~\ref{cor:ss}. 
\end{proof}

We turn to the proof of Theorem~\ref{thm:ss}.

\begin{proof}[Proof of Theorem \ref{thm:ss}]  We suppose Theorem~\ref{T:main} fails. Then Theorem~\ref{thm:apsol} gives us an almost periodic solution $u$ with modulation parameters $\xi(t)$ and $h(t)$ and compactness modulus $C(\eta)$.  Furthermore,  $u$ satisfies the conclusions of Corollary~\ref{cor:ss}. 

We first claim that $u$ is almost periodic with respect to parameters $\tilde\xi(t)\equiv \xi_0$ and $h(t)$ and compactness modulus $\tilde C(\eta)=C(\eta)+c(u)$ for some constant $c(u)$.  To see this, recall \eqref{eq:ap1} and note that by Corollary~\ref{cor:ss}(iii), we have
\[
|\xi - \xi(t)| \ge |\xi-\xi_0| - |\xi(t)-\xi_0| \ge |\xi-\xi_0| - c(u)h(t)
\]
for some $c(u)>0$. 

We can now define a solution $v$ via 
\[
e^{-it\Delta}v(t) = e^{-ix\xi_0}e^{-it\Delta}u(t),\qtq{that is,} v(t,x)=e^{-ix\xi_0-it|\xi_0|^2}u(t,x+2t\xi_0). 
\]
It follows that $v$ is almost periodic with parameters $h(t)$ and $\xi(t)\equiv 0$ and compactness modulus $\tilde C(\eta)$. 

As Corollary~\ref{cor:ss} implies $h(t)\sim_u t^{-\frac12}$ for $t\geq 1$, we can set $h(t)=t^{-\frac12}$, provided we modify $\tilde C(\eta)$ by another constant depending on $u$.

We have thus constructed an almost periodic solution $v:[1,\infty)\times\R^d\to\C$ that blows up forward in time, with modulation parameters $\xi(t)\equiv 0$ and $h(t)=t^{-\frac12}$. Equivalently, we can write
\[
[e^{-it\Delta} v(t)](x) = t^{-\frac{1}{p}} \psi(t, \tfrac{x}{\sqrt{t}})\qtq{for} t\geq 1,
\]
where $\psi$ takes values in a pre-compact set $K\subset\fhsc$. 

Now take a sequence $1\leq t_n\to\infty$ and define the normalized solutions
\[
v^{[t_n]}(t,x) := t_n^{\frac1{p}} v(t_nt,\sqrt{t_n}x).
\]
In particular, 
\[
v^{[t_n]}(1,x) = t_n^{\frac1{p}} v(t_n, \sqrt{t_n} x)= [e^{i\Delta} \psi(t_n)](x).
\]
Passing to a subsequence, we see that $v^{[t_n]}(1)\to e^{i\Delta}\psi_1$ strongly in $e^{i\Delta}\fhsc$ for some non-zero $\psi_1\in\fhsc$. 

Now let $\Psi$ be a solution to \eqref{nls} with data $e^{-i\Delta}\Psi(1)=\psi_1$.  We will show that $\Psi$ satisfies the conclusions of Theorem~\ref{thm:ss}.  That $\Psi$ is almost periodic on $I_{\max}\cap (0, \infty)$ with parameters $h(t)=t^{-\frac12}$ and $\xi(t)\equiv 0$ and compactness modulus $\tilde C(\eta)$ follows from the continuity of the data to solution map and a change of variables:
\begin{align*}
\int_{|x|\geq \tilde C(\eta)\sqrt{t}} \bigl||x|^{|s_c|}e^{-it\Delta}\Psi(t)\bigr|^2 \,dx & = \lim_{n\to\infty} \int_{|x|\geq \tilde C(\eta)\sqrt{t}}  \bigl||x|^{|s_c|}e^{-it\Delta}v^{[t_n]}(t)\bigr|^2\,dx \\
& = \lim_{n\to\infty} \int_{|x|\geq\tilde C(\eta)\sqrt{t_nt}}\bigl||x|^{|s_c|}e^{-it_nt\Delta}v(t_nt)\bigr|^2\,dx \leq\eta
\end{align*}
for any $\eta>0$ and $t>0$.  Similarly,
\[
\int_{|\xi|\geq\tilde C(\eta) t^{-\frac12}} \bigl|\,|\nabla|^{s_c}\F e^{-it\Delta}\Psi(t)\bigr|^2\,d\xi \leq \eta
\]
for any $\eta>0$ and $t>0$.  In particular, as $h(t)=t^{-\frac12}$, Proposition~\ref{prop:lc} implies that $I_{\max} \supseteq (0, \infty)$.  The fact that $h(t)$ blows up as $t\to 0$ shows that the solution cannot be extended continuously to $t=0$.  Thus $I_{\max} = (0, \infty)$.

To see that $\Psi$ blows up forward in time, we argue by contradiction: If $\|\Psi\|_{S([1, \infty))}<\infty$, then by Proposition~\ref{thm:lpt},
\[
\norm{v}_{S([t_n, \infty))}=\norm{v^{[t_n]}}_{S([1, \infty))} < \infty\quad \text{for $n$ sufficiently large},
\]
which contradicts the fact that $v$ blows up forward in time.  This completes the proof of Theorem~\ref{thm:ss}. 
\end{proof}

\section{Preclusion of self-similar solutions}\label{S:preclusion}

 The goal of this section is to prove the following:

\begin{theorem}[No self-similar solutions]\label{noss}  There are no self-similar almost periodic solutions to \eqref{nls}. 
\end{theorem}

Together with Theorem~\ref{thm:ss}, Theorem~\ref{noss} implies the main result, Theorem~\ref{T:main}.  We prove Theorem~\ref{noss} by contradiction. In particular, we show that self-similar solutions must belong to $L_x^2$.  However, by conservation of mass and $L_x^2$-subcriticality, this would imply that self-similar solutions are global, contradicting the fact that they blow up at $t=0$.

We begin by collecting some properties of self-similar solutions.  The first result is the following reduced Duhamel formula.

\begin{proposition}[Reduced Duhamel formula]\label{P:RD} Suppose $u$ is a self-similar almost periodic solution to \eqref{nls}. Then 
\[
\wlim_{t\to\infty} e^{-it\Delta}u(t) = \wlim_{t\to 0} e^{-it\Delta} u(t) = 0,
\]
where $\wlim$ indicates that the limit is taken in the weak topology on $\fhsc$.  In particular, for all $t\in (0,\infty)$,
\begin{align*}
e^{-it\Delta}u(t) = \wlim_{T\to\infty} i\mu\int_t^T e^{-is\Delta}\bigl(|u|^pu)(s)\,ds = -\wlim_{T\to 0} i\mu\int_T^t e^{-is\Delta}(|u|^p u)(s)\,ds.
\end{align*}
\end{proposition}

\begin{proof} It suffices to note that
\[
\lim_{t\to\infty} \langle e^{-it\Delta} u(t),\phi\rangle = \lim_{t\to 0}\langle e^{-it\Delta} u(t),\phi\rangle = 0
\]
for all test functions $\phi$, which follows from \eqref{ssK} and a change of variables. 
\end{proof}

Next, we show that almost periodic solutions obey uniform space-time bounds on dyadic intervals $[T,2T]\subset(0,\infty)$.

\begin{proposition}\label{P:boci} Suppose $u$ is a self-similar almost periodic solution to \eqref{nls}.  Then
\begin{equation}\label{E:boci}
\sup_{T\in(0,\infty)}\|u\|_{W(T,2T)} \lesssim_u 1.
\end{equation}
\end{proposition}

\begin{proof}  The argument leading to \eqref{eq:lc1} applies also in the current setting, yielding the following: For any $\eps>0$ there exists $\delta=\delta(u,\eps)$ so that
\begin{equation}\label{eq:lc1'}
\norm{e^{i(t-t_0)\Delta}u(t_0)}_{W([(1-\delta)t_0,(1+\delta)t_0] )} \le \eps\qtq{for any}t_0\in (0,\infty).
\end{equation}

Choosing $\eps>0$ sufficiently small, Corollary~\ref{C:stability} yields
$$
\norm{u}_{W([(1-\delta)t_0,(1+\delta)t_0] )} \lesssim \eps + \eps^\beta
$$
uniformly for $t_0\in (0,\infty)$.  The estimate \eqref{E:boci} now follows by covering $[T,2T]$ by $O(\delta^{-1})$ intervals of this type.
\end{proof}

We are now in a position to rule out self-similar solutions.  We modify arguments from \cite{KTV, KVZ}.  In our setting, the argument is considerably simpler due mainly to Lemma~\ref{L:ss1} below.

\begin{proof}[Proof of Theorem~\ref{noss}]  Throughout the proof, we denote the nonlinearity by
\[
F(u) = \mu |u|^p u.
\]
Suppose towards a contradiction that there exists a self-similar almost periodic solution  $u$.  Recalling the notation from \eqref{def:JN} and Section~\ref{S:FS}, we define 
\begin{align*}
 \M(A) & := \sup_{T\in(0,\infty)} \| J^{|s_c|}_{\leq AT^{\frac12}}(T)u(T)\|_{L_x^2}, \\
 \S(A) &  := \sup_{T\in(0,\infty)} \sum_{j=1}^2 \|J^{|s_c|}_{\leq AT^{\frac12}}(t)u(t)\|_{L_t^{q_j,2}L_x^{r_j}([T,2T]\times\R^d)}, \\
 \N(A) & := \sup_{T\in(0,\infty)} \| J^{|s_c|}_{\leq AT^{\frac12}}(t)F(u(t))\|_{L_t^{q_1',2} L_x^{r_1'}([T,2T]\times\R^d)}.
\end{align*}

As $1\leq \frac{t}{T}\leq 2$ for $t\in[T,2T]$ and $\M(A)$ is increasing in $A$, we have
\begin{equation}\label{infty}
\| J^{|s_c|}_{\leq AT^{\frac12}}(t)u(t)\|_{L_t^\infty L_x^2([T,2T]\times\R^d)} \lesssim \M(A),
\end{equation}
uniformly in $A,T$.  Also, by Proposition~\ref{P:boci},
\begin{equation}\label{E:msnbd}
\M(A)+\S(A) \lesssim_u 1.
\end{equation}

We prove three lemmas relating $\M,$ $\S$, and $\N$. 

\begin{lemma}[$\M$ controls $\S$]\label{L:ss1} There exists $\delta_0>0$ such that
\[
\S(A)\lesssim A^{\delta_0}\M(A).
\]
\end{lemma}

\begin{proof} By the dispersive estimate, H\"older's inequality, \eqref{infty}, and the fact that each $(q_j,r_j)$ is an admissible pair, we have
\begin{align*}
\| & J^{|s_c|}_{\leq AT^{\frac12}}(t)u(t) \|_{L_t^{q_j,2} L_x^{r_j}([T,2T]\times\R^d)} \\
& \lesssim \| |t|^{-\frac{2}{q_j}} \|_{L_t^{q_j,2}([T,2T])} \| \varphi(\tfrac{x}{A\sqrt{T}})|x|^{|s_c|} e^{-it\Delta}u(t)\|_{L_t^\infty L_x^{r_j'}([T,2T]\times\R^d)} \\
& \lesssim |T|^{-\frac{1}{q_j}}(AT^{\frac12})^{\frac{2}{q_j}} \|J^{|s_c|}_{\leq AT^{\frac12}}(t)u(t)\|_{L_t^\infty L_x^2([T,2T]\times\R^d)} \lesssim A^{\frac{2}{q_j}} \M(A)
\end{align*}
for $j=1,2.$ The result follows. 
\end{proof}

\begin{lemma}[$\S$ controls $\N$]\label{L:ss2} For $\eps>0$ small enough, 
\begin{equation}
\N(A) \lesssim_u \sum_{N>A} (\tfrac{A}{N})^{|s_c|+\eps} \S(N) + \bigl[\S(A)\bigr]^{p+1}. \label{macN}
\end{equation}
Consequently, for any $0<\delta<|s_c|+\eps$, we have
\begin{equation}\label{macN2}
\S(A)\lesssim_u A^\delta \implies \N(A)\lesssim_u A^\delta.
\end{equation}
\end{lemma}

\begin{proof}  We first prove \eqref{macN}. Fix $T>0$. In the estimates below, all space-time norms are taken over $[T,2T]\times\R^d$. Define 
\[
\tilde u(t) := M(-t)u(t).
\] 

Using \eqref{def:JN} and gauge invariance, we first note that
\[
\|J^{|s_c|}_{\leq AT^{\frac12}}F(u)\|_{L_t^{q_1',2}L_x^{r_1'}} \sim \| P_{\leq AT^{\frac12}t^{-1}} |t|^{|s_c|} |\nabla|^{|s_c|} F(\tilde u)\|_{L_t^{q_1',2} L_x^{r_1'}}.
\]
We decompose the nonlinearity as
\begin{equation}\label{E:decompose}
F(\tilde u) = F(\tilde u_{\leq A T^{\frac12}t^{-1}}) + [F(\tilde u)-F(\tilde u_{\leq A T^{\frac12}t^{-1}})].
\end{equation}

For the first term, we can estimate as in Lemma~\ref{L:nonlinear} and Lemma~\ref{L:embeddings} and use \eqref{def:Xn} to deduce
\[
\| |t \nabla|^{|s_c|} F(\tilde u_{\leq AT^{\frac12}t^{-1}}) \|_{L_t^{q_1',2}L_x^{r_1'}} \lesssim \S(A)^{p+1}.
\]

Next, we estimate the contribution of the second term in \eqref{E:decompose}.  We will apply Lemma~\ref{L:parap} with
\[
v(t) = -M(t) P_{>AT^{\frac12}t^{-1}}\tilde u(t).
\]
Note that by \eqref{def:Xn}, we have 
\[
\|v\|_{W}\lesssim \|u\|_W.
\]
Choosing $\eps>0$ sufficiently small, we use Bernstein, Lemma~\ref{L:parap}, and Proposition~\ref{P:boci} to estimate  
\begin{align*}
\bigl\| & |t \nabla|^{|s_c|} P_{\leq AT^{\frac12}t^{-1}} [F(\tilde u)-F(\tilde u_{\leq A T^{\frac12}t^{-1}})]  \bigr\|_{L_t^{q_1',2} L_x^{r_1'}} \\ 
& \lesssim (AT^{\frac12})^{|s_c|+\eps} \bigl\| |t \nabla|^{-\eps} [F(\tilde u)-F(\tilde u_{\leq A T^{\frac12}t^{-1}})] \bigr\|_{L_t^{q_1',2} L_x^{r_1'}} \\
& \lesssim (AT^{\frac12})^{|s_c|+\eps} \|u\|_W^p \| |t \nabla|^{-\eps} \tilde u_{>AT^{\frac12}t^{-1}} \|_{L_t^{q_1,2}L_x^{r_1}} \\
&\lesssim_u \sum_{N>A} (AT^{\frac12})^{|s_c|+\eps} \| | t \nabla|^{-\eps} P_{NT^{\frac12}t^{-1}} \tilde u\|_{L_t^{q_1,2}L_x^{r_1}}\lesssim_u\sum_{N>A}\bigl(\tfrac{A}{N}\bigr)^{|s_c|+\eps} \S(N).
\end{align*}
This completes the proof of \eqref{macN}. 

We turn to \eqref{macN2}. Suppose that $\S(A)\lesssim_u A^{\delta}$ for some $0<\delta<|s_c|+\eps$. Using \eqref{E:msnbd}, we estimate
\begin{align*}
\N(A) &\lesssim_u \sum_{N>A} (\tfrac{A}{N})^{|s_c|+\eps}N^{\delta} + \S(A)^p A^{\delta}  \lesssim_u A^\delta.
\end{align*}
This completes the proof of Lemma~\ref{L:ss2}. \end{proof}

\begin{lemma}[$\N$ controls $\M$]\label{L:ss3} For any $\delta>0$, 
\[
\N(A) \lesssim_u A^{\delta} \implies \M(A)\lesssim_u A^{\delta}.
\]
\end{lemma}

\begin{proof}  Fix $T>0$ and suppose we have the bound $\N(A)\lesssim_u A^{\delta}$. 

Using Proposition~\ref{P:RD} and the commutation properties of $J$,
\[
J^{|s_c|}_{\leq AT^{\frac12}} u(T) =\wlim_{T'\to\infty} i\int_T^{T'} e^{i(T-s)\Delta}J^{|s_c|}_{\leq AT^{\frac12}}(s) F(u(s))\,ds.
\]
Thus, by weak lower-semicontinuity of the norm and Strichartz,
\begin{align*}
\|J^{|s_c|}_{\leq AT^{\frac12}} u(T)\|_{L_x^2} & \leq \sum_{k=0}^\infty\,\biggl\| \int_{2^kT}^{2^{k+1}T} e^{i(T-s)\Delta}J^{|s_c|}_{\leq AT^{\frac12}}(s)F(u(s))\,ds\biggr\|_{L_x^2} \\
& \lesssim \sum_{k=0}^\infty \| J^{|s_c|}_{\leq AT^{\frac12}}(t) F(u(t))\|_{L_t^{q_1',2}L_x^{r_1'}([2^kT,2^{k+1}T]\times\R^d)} \\
& \lesssim \sum_{k=0}^\infty \N(2^{-\frac{k}{2}}A) 
 \lesssim_u \sum_{k=0}^\infty 2^{-\frac{\delta k}{2}} A^{\delta}  
 \lesssim_u A^\delta. 
\end{align*}
The result follows. 
\end{proof}

Using \eqref{E:msnbd} as a starting point, we can combine Lemma~\ref{L:ss1}, Lemma~\ref{L:ss2}, and Lemma~\ref{L:ss3} and iterate finitely many times  to deduce  
\[
\M(A)\lesssim_u A^{|s_c|+\eps}
\]
for some $\eps>0$. We claim that this implies $u(t)\in L_x^2$ for all $t\in(0,\infty)$.  

Indeed, given $t\in (0,\infty)$ and $A>0$, we have
\[
\| J^0_{\leq At^{\frac12}}(t)u(t)\|_{L_x^2} \lesssim \sum_{B\leq A} (B\sqrt{t})^{-|s_c|} \|J^{|s_c|}_{Bt^{\frac12}}(t)u(t)\|_{L_x^2} \lesssim_{u} A^{\eps}t^{-\frac{|s_c|}{2}}.
\]
On the other hand,
\[
\| J^0_{>At^{\frac12}}(t)u(t)\|_{L_x^2} \lesssim_u A^{-|s_c|}t^{-\frac{|s_c|}{2}}.
\]
Thus $u(t)\in L_x^2$.  Moreover, choosing $A=1$ and letting $t\to \infty$, we see that the $L^2_x$-norm of the solution decays to zero as $t\to \infty$.  By the conservation of mass, this implies that $u\equiv 0$, a contradiction.  This completes the proof of Theorem~\ref{noss}.
\end{proof}


\begin{thebibliography}{100}

\bibitem{Bar} J. Barab, \emph{Nonexistence of asymptotically free solutions for a nonlinear Schr\"odinger equation.}  J. Math. Phys. \textbf{25} (1984), no. 11, 3270--3273. MR0761850 

\bibitem{BahGer} H. Bahouri and P. G\'erard,  \emph{High frequency approximation of solutions to critical nonlinear wave equations.} Amer. J. Math. \textbf{121} (1999), no. 1, 131--175. MR1705001

\bibitem{Beg} P. B\'egout, \emph{Convergence to scattering states in the nonlinear Schr\"odinger equation.} Communications in Contemporary Mathematics \textbf{3}, (2001), 403--418. MR1849648

\bibitem{BegVar} P. B\'egout and A. Vargas, \emph{ Mass concentration phenomena for the $L^2$-critical nonlinear Schr\"odinger equation.} Trans. Amer. Math. Soc. \textbf{359} (2007), no. 11, 5257--5282. MR2327030

\bibitem{Bou1} J. Bourgain, \emph{Global wellposedness of defocusing critical nonlinear Schr\"odinger equation in the radial case.} J. Amer. Math. Soc. \textbf{12} (1999), no. 1, 145--171. MR1626257

\bibitem{Bou2} J. Bourgain, \emph{Refinements of Strichartz' inequality and applications to 2D-NLS with critical nonlinearity.} Internat. Math. Res. Notices 1998, no. 5, 253--283. MR1616917

\bibitem{CarKer} R. Carles and S. Keraani, \emph{On the role of quadratic oscillations in nonlinear Schr\"odinger equations. II. The $L^2$-critical case.} Trans. Amer. Math. Soc. \textbf{359} (2007), no. 1, 33--62.  MR2247881

\bibitem{Caz} T. Cazenave, \emph{Semilinear Schr\"odinger equations.} Courant Lecture Notes in Mathematics, 10. New York University, Courant Institute of Mathematical Sciences, New York; American Mathematical Society, Providence, RI, 2003. MR2002047

\bibitem{CazWei} T. Cazenave and F. B. Weissler, \emph{The Cauchy problem for the critical nonlinear Schr\"odinger equation in $H^s$.} Nonlinear Anal. \textbf{14} (1990), no. 10, 807--836.  MR1055532

\bibitem{CKSTT} J. Colliander, M. Keel, G. Staffilani, T. Takaoka, and T. Tao, \emph{Global well-posedness and scattering for the energy-critical nonlinear Schr\"odinger equation in $\R^3$.} Ann. of Math. (2) \textbf{167} (2008), no. 3, 767--865. MR2415387

\bibitem{ChoHwaOza} Y. Cho, G. Hwang, and T. Ozawa, \emph{Global well-posedness of critical nonlinear Schr\"odinger equations below $L^2$.} Discrete Contin. Dyn. Syst. \textbf{33} (2013), no. 4, 1389--1405.  MR2995852

\bibitem{ChrColTao} M. Christ, J. Colliander, and T. Tao, \emph{A priori bounds and weak solutions for the nonlinear Schr\"odinger equation in Sobolev spaces of negative order.} J. Funct. Anal. \textbf{254} (2008), no. 2, 368--395. MR2376575

\bibitem{ChrWei} M. Christ and M. Weinstein, \emph{Dispersion of small amplitude solutions of the generalized Korteweg-de Vries equation.} J. Funct. Anal. \textbf{100} (1991), no. 1, 87--109. MR1124294

\bibitem{Dod1} B. Dodson, \emph{Global well-posedness and scattering for the defocusing, $L^2$-critical, nonlinear Schr\"odinger equation when $d=1$.} Amer. J. of Math. \textbf{138} (2016), no. 2, 531--569.

\bibitem{Dod2} B. Dodson, \emph{Global well-posedness and scattering for the defocusing, $L^2$-critical, nonlinear Schr\"odinger equation when $d=2$.} Preprint {\tt arXiv:1006.1375}

\bibitem{Dod3} B. Dodson, \emph{Global well-posedness and scattering for the defocusing, $L^2$-critical nonlinear Schr\"odinger equation when $d\geq 3$.}  J. Amer. Math. Soc. \textbf{25} (2012), no. 2, 429--463. MR2869023

\bibitem{Dod4} B. Dodson, \emph{Global well - posedness and scattering for the focusing, energy - critical nonlinear Schr\"odinger problem in dimension $d=4$ for initial data below a ground state threshold.} Preprint {\tt arXiv:1409.1950}

\bibitem{Dod5} B. Dodson, \emph{Global well-posedness and scattering for the mass critical nonlinear Schr\"odinger equation with mass below the mass of the ground state.} Adv. Math. \textbf{285} (2015), 1589--1618. MR3406535

\bibitem{DMMZ} B. Dodson, C. Miao, J. Murphy, and J. Zheng, \emph{The defocusing quintic NLS in four space dimensions.} Preprint {\tt arXiv:1508.07298}

\bibitem{GinVel} J. Ginibre and G. Velo, \emph{Smoothing properties and retarded estimates for some dispersive evolution equations.} 
Comm. Math. Phys. \textbf{144} (1992), no. 1, 163--188. MR1151250

\bibitem{Gri} M. Grillakis, \emph{On nonlinear Schr\"odinger equations.} Comm. Partial Differential Equations \textbf{25} (2000), no. 9--10, 1827--1844. MR1778782

\bibitem{Hid} K. Hidano, \emph{Nonlinear Schr\"odinger equations with radially symmetric data of critical regularity.} Funkcial. Ekvac. \textbf{51} (2008), no. 1, 135--147. MR2428826

\bibitem{Hunt} R. A. Hunt, \emph{On L(p,q) spaces.} 
Enseignement Math. (2) \textbf{12} (1966) 249--276. MR0223874

\bibitem{KeeTao} M. Keel and T. Tao, \emph{Endpoint Strichartz estimates.} Amer. J. Math. \textbf{120} (1998), no. 5, 955--980. MR1646048

\bibitem{KenMer} C. E. Kenig and F. Merle, \emph{Global well-posedness, scattering and blow-up for the energy-critical, focusing, non-linear Schr\"odinger equation in the radial case.} Invent. Math. \textbf{166} (2006), 645--675. MR2257393

\bibitem{KenMer2} C. E. Kenig and F. Merle, \emph{Scattering for $\dot H^{1/2}$ bounded solutions to the cubic, defocusing NLS in 3 dimensions.} Trans. Amer. Math. Soc. \textbf{362} (2010), no. 4, 1937--1962.  MR2574882

\bibitem{Ker1} S. Keraani, \emph{On the defect of compactness for the Strichartz estimates of the Schr\"odinger equations.} J. Differential Equations \textbf{175} (2001), no. 2, 353--392. MR1855973

\bibitem{Ker2} S. Keraani, \emph{On the blow up phenomenon of the critical nonlinear Schr\"odinger equation.} J. Funct. Anal. \textbf{235} (2006), no. 1, 171--192.  MR2216444

\bibitem{KTV} R. Killip, T. Tao, and M. Visan, \emph{The cubic nonlinear Schr\"odinger equation in two dimensions with radial data.} J. Eur. Math. Soc. (JEMS) \textbf{11} (2009), no. 6, 1203--1258. MR2557134

\bibitem{KV1} R. Killip and M. Visan, \emph{Nonlinear Schr\"odinger equations at critical regularity.} Clay Math. Proc. \textbf{17} (2013), 325--437. MR3098643

\bibitem{KV2} R. Killip and M. Visan, \emph{The focusing energy-critical nonlinear Schr\"odinger equation in dimensions five and higher.} Amer. J. Math. \textbf{132} (2010), no. 2, 361--424.  MR2654778

\bibitem{KV3} R. Killip and M. Visan, \emph{Global well-posedness and scattering for the defocusing quintic NLS in three dimensions.} Anal. PDE \textbf{5} (2012), no. 4, 855--885. MR3006644

\bibitem{KV4} R. Killip and M. Visan, \emph{Energy-supercritical NLS: critical $\dot H^s$-bounds imply scattering.} Comm. Partial Differential Equations \textbf{35} (2010), no. 6, 945--987.  MR2753625

\bibitem{KVZ} R. Killip, M. Visan, and X. Zhang, \emph{The mass-critical nonlinear Schr\"odinger equation with radial data in dimensions three and higher.} Anal. PDE \textbf{1} (2008), no. 2, 229--266. MR2472890

\bibitem{Mas1} S. Masaki, \emph{On minimal non-scattering solution for focusing mass-subcritical nonlinear Schr\"odinger equation.} Preprint {\tt arXiv:1301.1742}
 
\bibitem{Mas2} S. Masaki, \emph{A sharp scattering condition for focusing mass-subcritical nonlinear Schr\"odinger equation.} Commun. Pure Appl. Anal. \textbf{14} (2015), no. 4, 1481--1531. MR3359531

\bibitem{Mas3} S. Masaki and J. Segata, \emph{Existence of a minimal non-scattering solution to the mass-subcritical generalized Korteweg-de Vries equation.} Preprint {\tt arXiv:1602.05331}

\bibitem{Mas4} S. Masaki, \emph{Two minimization problems on non-scattering solutions to mass-subcritical nonlinear Schr\"odinger equation.} Preprint \texttt{arXiv:1605.09234}

\bibitem{MasSeg} S. Masaki and J. Segata, \emph{On well-posedness of the generalized Korteweg-de Vries equation in scale critical $\hat{L}^r$ space}.  To appear in Analysis and PDE.

\bibitem{MerVeg} F. Merle and L. Vega, \emph{Compactness at blow-up time for $L^2$ solutions of the critical nonlinear Schr\"odinger equation in 2D.} Internat. Math. Res. Notices 1998, no. 8, 399--425. MR1628235

\bibitem{MMZ} C. Miao, J. Murphy, and J. Zheng, \emph{The defocusing energy-supercritical NLS in four space dimensions.} J. Funct. Anal. \textbf{267} (2014), no. 6, 1662--1724.  MR3237770

\bibitem{Mur1} J. Murphy, \emph{Intercritical NLS: critical $\dot H^s$-bounds imply scattering.} SIAM J. Math. Anal. \textbf{46} (2014), no. 1, 939--997. MR3166962

\bibitem{Mur2} J. Murphy, \emph{The defocusing $\dot H^{1/2}$-critical NLS in high dimensions.} Discrete Contin. Dyn. Syst. \textbf{34} (2014), no. 2, 733--748. MR3094603

\bibitem{Mur3} J. Murphy, \emph{The radial defocusing nonlinear Schr\"odinger equation in three space dimensions.} Comm. Partial Differential Equations \textbf{40} (2015), no. 2, 265--308.  MR3277927

\bibitem{Nak} K. Nakanishi, \emph{Asymptotically-free solutions for the short-range nonlinear Schr\"odinger equation.} SIAM J. Math. Anal. \textbf{32} (2001), no. 6, 1265--1271. MR1856248

\bibitem{NakOza} K. Nakanishi and T. Ozawa, \emph{Remarks on scattering for nonlinear Schr\"odinger equations.} NoDEA Nonlinear Differential Equations Appl. \textbf{9} (2002), no. 1, 45--68.  MR1891695

\bibitem{One} R. O'Neil, \emph{Convolution operators and $L(p,q)$ spaces.} Duke Math. J. \textbf{30} (1963), 129--142.  MR0146673

\bibitem{RV} E. Ryckman and M. Visan, \emph{Global well-posedness and scattering for the defocusing energy-critical nonlinear Schr\"odinger equation in ℝ1+4.} Amer. J. Math. \textbf{129} (2007), no. 1, 1--60. MR2288737

\bibitem{Str} W. Strauss, \emph{Nonlinear scattering theory}, Scattering Theory in Math. Physics, Reidel, Dordrecht, 1974, pp. 53--78. 

\bibitem{Strichartz} R. Strichartz, \emph{Restrictions of Fourier transforms to quadratic surfaces and decay of solutions of wave equations.} Duke Math. J. \textbf{44} (1977), no. 3, 705--714. MR0512086

\bibitem{Tao1} T. Tao, \emph{Global well-posedness and scattering for the higher-dimensional energy-critical nonlinear Schr\"odinger equation for radial data.} New York J. Math. \textbf{11} (2005), 57--80.  MR2154347

\bibitem{TVZ} T. Tao, M. Visan, and X. Zhang, \emph{Global well-posedness and scattering for the defocusing mass-critical nonlinear Schr\"odinger equation for radial data in high dimensions.} Duke Math. J. \textbf{140} (2007), no. 1, 165--202. MR2355070

\bibitem{TsuYaj} Y. Tsutsumi and K. Yajima, \emph{The asymptotic behavior of nonlinear Schr\"odinger equations.} Bull. Amer. Math. Soc. (N.S.) \textbf{11} (1984), no. 1, 186--188.  MR0741737

\bibitem{Vis0} M. Visan, \emph{The defocusing energy-critical nonlinear Schr\"odinger equation in dimensions five and higher.} Ph.D Thesis, UCLA, 2006. MR2709575

\bibitem{Vis1} M. Visan, \emph{The defocusing energy-critical nonlinear Schr\"odinger equation in higher dimensions.} Duke Math. J. \textbf{138} (2007), no. 2, 281--374.  MR2318286

\bibitem{Vis2} M. Visan, \emph{Global well-posedness and scattering for the defocusing cubic nonlinear Schr\"odinger equation in four dimensions.} Int. Math. Res. Not. IMRN 2012, no. 5, 1037--1067. MR2899959

\bibitem{Vis3} M. Visan, \emph{Dispersive Equations.} In ``Dispersive Equations and Nonlinear Waves'', Oberwolfach
Seminars \textbf{45}, Birkh\"auser/Springer Basel 2014.

\bibitem{XieFan} J. Xie and D. Fang, \emph{Global well-posedness and scattering for the defocusing $\dot H^s$-critical NLS.} Chin. Ann. Math. Ser. B \textbf{34} (2013), no. 6, 801--842. MR3122297

\end{thebibliography}
\end{document}